\documentclass{amsart}
\PassOptionsToPackage{prologue,dvipsnames}{xcolor}
\usepackage[utf8]{inputenc}
\usepackage{amsfonts}
\usepackage{hyperref}
\usepackage{dynkin-diagrams}
\usepackage{tikz-cd} 
\usepackage{verbatim}
\usepackage{cite}
\usepackage{varwidth}
\usepackage{float}
\usepackage{physics}
\usepackage{array,longtable}
\newcolumntype{C}{>{$}c<{$}} 
\newcolumntype{L}{>{$}l<{$}} 

\definecolor{yelloworange}{RGB}{248, 186, 0}

\usepackage{amsmath}
\usepackage{amsthm}
\usepackage{pdflscape}
\usepackage{pgfplots}
\pgfplotsset{compat=1.18}
\usepackage{mathrsfs}
\usepackage{tikz}
\usetikzlibrary{braids,calc,shapes.geometric,matrix, quotes, angles,intersections, chains, arrows.meta, shapes, positioning,
arrows,shapes,positioning, through, backgrounds}
\tikzset{suspend join/.code={\def\tikz@after@path{}}}

\newtheorem{thm}{Theorem}[section]
\newtheorem{cor}[thm]{Corollary}
\newtheorem{prop}[thm]{Proposition}
\newtheorem{lem}[thm]{Lemma}

\theoremstyle{definition}
\newtheorem{obs}[thm]{Observation}
\newtheorem{defn}[thm]{Definition}

\newtheorem{algm}[thm]{Algorithm}
\newtheorem{exmp}[thm]{Example}

\newtheorem{construction}[thm]{Construction}

\newtheorem{rmk}[thm]{Remark}

\theoremstyle{remark}

\makeatletter
\let\c@equation\c@thm
\makeatother
\numberwithin{equation}{section}

\newcommand{\Z}{\mathbb{Z}}
\newcommand{\F}{\mathbb{F}}
\newcommand{\Q}{\mathbb{Q}}
\newcommand{\PDS}{perfect difference set}
\newcommand{\llrrparen}[1]{% \llrrparen{..}
  \left(\mkern-2mu\left(#1\right)\mkern-2mu\right)}

\DeclareMathOperator\PGL{PGL}

\def\MyColorList{{"red","cyan","blue"}}
\newcommand{\midarrow}{\tikz \draw[->] (0,0) -- +(.1,0);}

\def\centerarc[#1](#2)(#3:#4:#5)% Syntax: [draw options] (center) (initial angle:final angle:radius)
    { \draw[#1] ($(#2)+({#5*cos(#3)},{#5*sin(#3)})$) arc (#3:#4:#5); }
    
\title{Triangle Presentations Encoded by Perfect Difference Sets}

\author{Amy Herron}

\begin{document}

\begin{abstract}
When James Singer exhibited projective planes for all prime power orders in 1938, he realized these using the trace function of cubic extensions of a finite field and linked $\text{trace}=0$ to perfect difference sets.  In 1993, Cartwright, Mantero, Steger, and Zappa found that this trace function can be used to create a triangle presentation, which determines the structure of an $\tilde{A}_2$ building.  We demonstrate a new, intrinsic connection between the perfect different sets of Singer and the triangle presentations of Cartwright et al., and show that this connection improves the efficiency of algorithms that generate these triangle presentations. Moreover, we translate the panel-regular groups of Essert \cite{essert2013geometric} and Witzel \cite{witzel2017panel} using triangle presentation nomenclature.  This translation creates a uniform understanding of the panel-regular groups and vertex-regular groups via triangle presentations. 

\end{abstract}

\maketitle
\tableofcontents

\section*{Introduction}

In 1906, Veblen and Bussey proved that finite projective planes exist for all prime power orders \cite{veblen1906finite}.  Subsequently in 1938, Singer realized these projective planes using the trace function of cubic extensions of a finite field and connected $\text{trace}=0$ to  perfect difference sets.  Thus he proved one direction of the long-standing prime power conjecture:  an integer is the order of a projective plane if and only if it is a prime power.  

An $\tilde A_2$ building is a simplicial complex whose vertex links are all incidence graphs of some fixed projective plane.  In \cite{cartwright1993groups}, Cartwright, Mantero, Steger, and Zappa construct a group $\Gamma$ of type-rotating automorphisms of an $\tilde A_2$ building that acts simply transitively on the vertices of the building, and has generators in one-to-one correspondence with the points of the fixed projective plane.
The authors then found a construction for any prime power $q$ for a specific $\Gamma$, say $\Gamma_0$, that they call ``of Tits type." This $\Gamma_0$ embeds as an arithmetic subgroup of $\PGL\big(3,\F_q\llrrparen{t}\big)$. 

This construction makes use of Singer's construction of projective planes by identifying points of the projective plane with $\F_{q^3}^\times / \F_q^\times$.  These points can be considered as the first $q^2+q+1$ points of the cyclic group $\F_{q^3}^\times$ with primitive element $\zeta$ and are usually referred to by their power (e.g., $\zeta^i$ is referred to as point $i$). 
They then define a bijective function $\alpha: \textit{Points}\to\textit{Lines}$ such that point $0$ is sent to the line consisting of all the $(q+1)$ powers $i$ such that $tr_{\F_{q^3}^\times / \F_q^\times}(\zeta^i)=0$.  Subsequent points, say $x$, are sent to the lines 
$\{y\in \textit{Points}: tr_{\F_{q^3}^\times / \F_q^\times}(\zeta^y/\zeta^x)=0\}$.  Function $\alpha$ is used to determine the relations of $\Gamma_0$.
%From there, they have a formula to determine the relations of $\Gamma_0$ using $\alpha$.  

The Cayley graph of $\Gamma_0$ ends up being the 1-skeleton of the $\tilde{A}_2$ building.  The relations of $\Gamma_0$ come from triples of generators that make a triangle in the Cayley graph. The set of corresponding triples of the points in the projective plane is called a \textit{triangle presentation}.  

We introduce a way to construct a group isomorphic to $\Gamma_0$ using only the elements from a perfect difference set.  That is, any \PDS\ of order $q$ that is invariant by multiplication by $q$ encodes all the information needed to obtain the triangle presentation (see Theorem \ref{thm:main thm}).  Our method also significantly reduces the complexity of generating examples of triangle presentations compatible with bijective map $\alpha$.  (See Algorithm \ref{algorithm}.)  

Moreover, we show that multiple triangle presentations constructed in this manner are associated with a single given bijective function $\alpha$ (see Proposition \ref{prop:2^k}) and that all such triangle presentations are equivalent.  An important step in showing this is to prove that projective planes described by any two particular types of \PDS s can be transformed to one another via a collineation or correlation (see Theorem \ref{thm: PDS collineation or correlation}).

Both Essert and Witzel constructed groups of type-preserving automorphisms of an $\tilde A_2$ building that act simply transitively on the \textit{edges} of the building \cite{essert2013geometric} and \cite{witzel2017panel}.  In doing so, they also make use of Singer's construction of projective planes and connect this construction to \PDS s.  In \cite{witzel2017panel}, Witzel finds a subgroup of these groups that also acts simply transitively on each vertex type.  This subgroup turns out to be the intersection of our group $\Gamma$ with the ones created by \cite{essert2013geometric} and \cite{witzel2017panel}. We give a new description of this subgroup via triangle presentations and create a uniform understanding of our $\Gamma$ and the panel-regular groups in Section \ref{subsection:extension}.

Section \ref{sec:preliminaries} provides the necessary background.  
Section \ref{sec: main thm} proves the main theorem (Theorem \ref{thm:main thm}) that perfect difference sets encode triangle presentations for the aforementioned arithmetic subgroup of $\PGL\big(3,\F_q\llrrparen{t}\big)$. This section also looks at the equivalence of triangle presentations, creates an efficient algorithm to construct $\mathscr{T}$, and explains the extension of $\Gamma$ by automorphisms using this new construction.  Section \ref{subsection:extension} discusses the connection to the work of Essert and Witzel.  
Section \ref{sec:examples} gives several examples of the correspondence between perfect difference sets and triangle presentations.

\hfill\\

\section*{Acknowledgements}

We thank UNCG Summer School 2019: Computational Aspects of Buildings (DMS 1802448) for introducing us to the subject area and the Simons Foundation (965204, JM) for much travel support. Parts of this work were completed while the author was supported by the University of Giessen, particularly Stefan Witzel.  We further thank Stefan Witzel for his interest from the beginning of our ideas for this paper and various discussions that followed.  %, particularly the suggestion to add the discussion of the group extension in Section \ref{subsection:extension}. 
Much thanks goes to Pierre-Emmanuel Caprace and Anne Thomas for providing invaluable conversations and edits.  Lastly, we cannot thank Johanna Mangahas enough for painstakingly working with us throughout the research and writing of this paper.

\section{Preliminaries} \label{sec:preliminaries}

The first three sections provide the necessary background to understand what triangle presentations are.  Because they are exclusive to $\tilde A_2$ buildings, whose vertex links are the incidence graph of a projective plane, Section \ref{subsection:projplane} begins with defining \textit{projective planes} and includes concepts relevant to the understanding of triangle presentations. Section \ref{subsection:buildings} defines \textit{buildings} and related concepts such as Coxeter complexes and links of vertices that are specific to $\tilde A_2$ buildings.  Section \ref{subsection:tripres} introduces \textit{triangle presentations}.  Lastly, Section \ref{subsection:PDS} defines \textit{\PDS s} and introduces relevant material needed for Theorem \ref{thm:main thm}.

\subsection{Projective Planes} \label{subsection:projplane} \hfill\\

This section covers basic definitions and facts about finite projective planes, most of which can be found in any standard reference about finite projective planes (see, for example \cite{PP_Stevenson}).  Henceforth, ``projective planes" will refer exclusively to finite projective planes. 

\begin{defn} 
    A \textit{projective plane} is a triple of sets of points, lines, and incidence relations that satisfy the following axioms:
    \begin{enumerate}
        \item Every two points lie on a unique line,
        \item every two lines intersect at a unique point, and
        \item there exist four points, no three of which are collinear.  
    \end{enumerate}
\end{defn}

\begin{defn}
    The \textit{order} $q$ of a projective plane is the number of points on a line minus 1.  
\end{defn}

Some well-known facts that can be derived from the definition of a projective plane are as follows:
\begin{enumerate}
    \item  The number of points on a line equals the number of lines through    a point.
    \item A projective plane has the same number of points as lines.
    \item Every line in a projective plane contains the same number of points.  
    \item A projective plane of order $q$ has $q^2+q+1$ points/lines.
\end{enumerate}

Veblen and Bussey \cite{veblen1906finite} constructed examples of projective planes associated with all finite fields $\F_q$, where $q$ is necessarily a prime power and is also the order of these projective planes. These are called \textit{Desarguesian} projective planes, and they are the kind corresponding to the $\tilde A_2$-buildings studied here.

The smallest projective plane has order 2 and was discovered by Gino Fano\cite{fano}.  It can be thought of as a unit cube with hyperplanes intersecting four corners, one of which is the origin.  Figure \ref{fig:fano} is an example picture of the Fano plane with points labeled as coordinates of the unit cube.

\begin{figure}[h!]
\begin{center}
	\begin{tikzpicture}[scale=0.6]
		\tikzstyle{point}=[circle, draw=black, fill=white, inner sep=0.1cm]
		\node (v7) at (0,0) [point] {111};
			\draw (0,0) circle (2cm);
		\node (v1) at (90:4cm) [point] {010};
		\node (v2) at (210:4cm) [point] {100};
		\node (v4) at (330:4cm) [point] {001};
		\node (v3) at (150:2cm) [point] {110};
		\node (v6) at (270:2cm) [point] {101};
		\node (v5) at (30:2cm) [point] {011};
			\draw (v1) -- (v3) -- (v2);
			\draw (v2) -- (v6) -- (v4);
			\draw (v4) -- (v5) -- (v1);
			\draw (v3) -- (v7) -- (v4);
			\draw (v5) -- (v7) -- (v2);
			\draw (v6) -- (v7) -- (v1);
	\end{tikzpicture}
\end{center}
    \caption{Fano Plane} 
    \label{fig:fano}
\end{figure}

In section \ref{sec: equiv tri pres}, we will be looking at ways to create new projective planes from a given projective plane.  There are two ways to do this---via collineation and correlation as defined below.  

\begin{defn}
    A \textit{collineation} of a projective plane is a bijection that sends points to points and lines to lines such that collinear points in the domain are also collinear in the image.    
\end{defn}

\begin{defn}

    A \textit{correlation} of a projective plane is a bijection that maps points to lines, and lines to points while reversing incidence. Say $\beta$ is such a bijection.  Then if point $p$ is on line $L$, $\beta(L)$ is on line $\beta(p)$.
    
\end{defn}

\hfill\\
\\

\subsection{$\tilde{A}_2$ Buildings} \label{subsection:buildings} \hfill\\
\par This section covers basic definitions of spherical and affine buildings as well as provides constructions specific to $A_2$ and $\tilde A_2$ buildings.  We motivate this section with the construction of an $A_2$ building via the incidence graph of a projective plane.

\par Figure \ref{fig:Heawood} shows the incidence graph of the projective plane over field $\F_2$ from Figure \ref{fig:fano}.  The white vertices represent points of the projective plane and gray vertices represent lines of the projective plane.  Edges are determined by incidence (i.e., if a point is on a line or a line contains a point).

\begin{figure}[h!]
    \centering
    \includegraphics[width=0.5\linewidth]{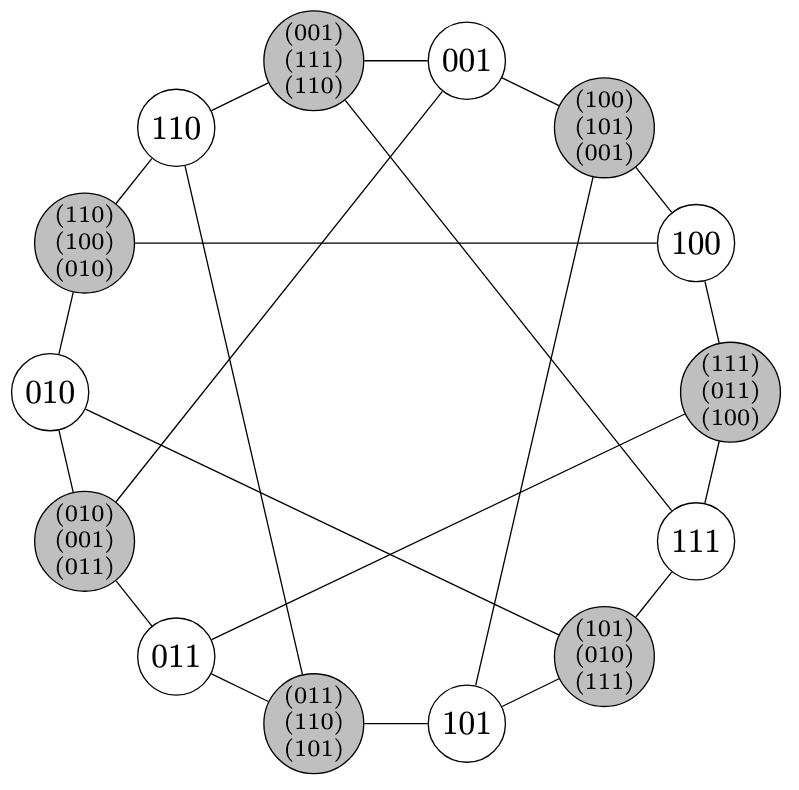}
    \caption{Bipartite Graph of Fano Plane}
    \label{fig:Heawood}
\end{figure}

The incidence graph of a projective plane over field $\F_q$ contains vertices that represent the points of the projective plane and vertices that represent the lines. The span of the vector associated with the points has dimension one over $\F_q$ and that the span of the vectors associated with the lines has dimension two over $\F_q$.  The girth of the bipartite graph is of size 6.  And every cycle represents a basis (via the points) for $(\F_{q})^3$. These incidence graphs are examples of a \textit{building} of type $A_2$.  Before the formal definition of a building, we will define Coxeter groups, Coxeter systems, standard cosets, and Coxeter complexes, the latter of which is used in the definition of a \textit{building}. 

\begin{defn}
    A \textit{Coxeter group} is a group with presentation $\langle s_i, i\in I \ | \ (s_is_j)^{m_{ij}}=1 \rangle$ for $I$ an index set and $m_{ij}\in \Z^+\cup \{\infty\}$ where $m_{ij}=1$ if $i=j$ and $m_{ij}=m_{ji}\geq 2$ if $i\neq j$.  If $m_{ij}=\infty$, then there is no relation between $s_i$ and $s_j$.  
\end{defn}

Two examples relevant to this paper are $A_2=\langle s_1,s_2 \ | \ (s_is_i)=(s_1s_2)^3 = 1 \rangle$ and $\tilde A_2 = \langle s_1,s_2,s_3 : (s_is_j)^3=s_is_i=1 \text{ for } i,j=1,2,3 \text{ and } i\neq j\rangle$.  Note that $A_2$ is a finite Coxeter group (called \textit{spherical}) and $\tilde A_2$ is an infinite Coxeter group (called \textit{affine} or \textit{Euclidean}).

\begin{defn}
    Let $W$ be the Coxeter group and $S$ its set of generators.  Then $(W,S)$ is a \textit{Coxeter system}.  Because isomorphic Coxeter groups need not have the same generators, the Coxeter system informs us of the generators that we are using to construct the Coxeter group.
\end{defn}

\begin{defn} [Definition 2.12, \cite{abramenko2008buildings}]
    Let $(W,S)$ be a Coxeter system and let $T\subseteq S$.  Then the group generated by $T$ is a \textit{standard subgroup} of $W$ and $w<T>$ is a \textit{standard coset} of $<T>$ for any $w\in W$.  
\end{defn}

\begin{defn} \label{defn: Coxeter} [Definition 3.1, \cite{abramenko2008buildings}]
    For every Coxeter system $(W,S)$, we can associate a simplicial complex $\Sigma(W,S)$ as follows: Let $(W,S)$ be a Coxeter system and let $\Sigma(W,S)$ be the poset of standard cosets in $W$, ordered by reverse inclusion.  Thus if $A\preceq B$ in $\Sigma$, then $B\subseteq A$.  We call $\Sigma(W,S)$ the \textit{Coxeter complex} associated to $(W,S)$.  
\end{defn}

For spherical Coxeter groups, the generators can be realized as hyperplanes in Euclidean space over which elements of the space reflect.  Consider the generators' intersection to be the origin.  The angle of intersection of any two generators $s_i, s_j$ is $\frac{\pi}{m_{ij}}$ .  For example, the intersection angle between $s_1$ and $s_2$ 
of the Coxeter group $A_2$ is $\frac{\pi}{3}$.  We can reflect over hyperplanes to generate the Coxeter complex (see Figure \ref{fig:A_2}). 

For Euclidean Coxeter groups, the generators can be realized as hyperplanes in affine space over which elements of the space reflect.  Consider the Coxeter group $\tilde{A}_2$.  The intersection angle between every $s_i$ and $s_j$ is $\frac{\pi}{3}$.  This time, when drawing the hyperplanes we get an intersection that is an equilateral triangle, which we can call a fundamental domain.  Without loss of generality, we can consider that the intersection of hyperplanes $s_1$ and $s_2$ is at the origin.  This makes the hyperplane $s_3$ an affine reflection. (See Figure \ref{fig:s123})

When we keep reflecting the triangle and its corresponding images across the hyperplanes, we tessellate the Euclidean plane by equilateral triangles.  And every line formed by edges of triangles can be considered another hyperplane, all of which will be parallel to the initial hyperplanes.  We can pick any two intersecting hyperplanes and see that the link of the intersection point is a Coxeter complex of type $A_2$. (See Figure \ref{fig:tessCoxComp}.)

\begin{figure}[h!]
\begin{center}
%\begin{minipage}{60mm}
    \begin{tikzpicture}[scale=0.5]
        \coordinate (A) at (0,0);
        \coordinate (B) at (4,0);
        %[label={[blue]left:$D$}]
        \coordinate[label={[red] above:{$s_1$}}] (C) at (2,3.464);
        \coordinate[label={[blue] above:{$s_2$}}] (D) at (-2,3.464);
        \coordinate (E) at (-4,0);
        \coordinate (F) at (-2,-3.464);
        \coordinate (G) at (2,-3.464);
    
        \draw[draw opacity=0, fill=gray!30] (0,0) -- (2,3.464) 
            arc[start angle=60, end angle=120,radius=4cm] -- (0,0);
            
        \draw[draw opacity=0, fill opacity=0.3, fill=blue!30] (0,0) -- (-2,3.464) 
            arc[start angle=120, end angle=180,radius=4cm] -- (0,0);
    
        \draw[draw opacity=0, fill opacity=0.3, fill=red!30] (0,0) -- (4,0) 
            arc[start angle=0, end angle=60,radius=4cm] -- (0,0);

        \draw[draw opacity=0, fill opacity=0.3, fill=Plum!30] 
            (0,0) -- (-4,0) arc[start angle=180, end angle=240,radius=4cm]
            -- (0,0);
            
        \draw[draw opacity=0, fill opacity=0.3, fill=Plum!30] 
            (0,0) -- (4,0) arc[start angle=0, end angle=-60,radius=4cm]
            -- (0,0);
           
         \draw[draw opacity=0, fill opacity=0.5, fill=lightgray!30] 
            (0,0) -- (-2,-3.464) arc[start angle=240, end angle=300,radius=4cm] 
            -- (0,0);
   
        \draw (C) -- (A) -- (D)  
            pic["$\frac{\pi}{3}$", text=darkgray, draw=darkgray, ->, angle radius=1cm]
            {angle=C--A--D};
            
        \centerarc[red,thick,<->]({2.8*cos(60)}, {2.8*sin(60)})(10:110:0.73)
        \centerarc[blue,thick,<->]({2.8*cos(120)}, {2.8*sin(120)})(70:170:0.73)

        \node at (0,2.5) {$1$};
        \node at ({2.5*cos(150)}, {2.5*sin(150)}) { $s_2$};
        \node at ({2.5*cos(30)}, {2.5*sin(30)}) { $s_1$};
        \node at ({2.5*cos(210)}, {2.5*sin(210)}) { $s_2s_1$};
        \node at ({2.5*cos(-30)}, {2.5*sin(-30)}) { $s_1s_2$};
        \node at (0,-2) { $s_1s_2s_1$};
        \node at (0,-2.5) {=};
        \node at (0,-3) { $s_2s_1s_2$};

        \draw (A) edge[gray, dashed, thick,->] (B);
        \draw (A) edge[thick, draw=red,->] (C);
        \draw (A) edge[thick, draw=blue,->] (D);
        \draw (A) edge[gray, dashed, thick,->] (E);
        \draw (A) edge[thick, draw=red,->] (F);
        \draw (A) edge[thick, draw=blue,->] (G);
    \end{tikzpicture}
\end{center}
    \caption{Coxeter Complex $A_2$}
    \label{fig:A_2}
\end{figure}

% \begin{figure}[h!]
% \begin{center}
% \begin{tikzpicture}[scale=0.5]

%         \coordinate (A) at (0,0);
%         \coordinate (B) at (5,2);
%         \coordinate[label={[red] above:{$s_1$}}] (C) at ({6*cos(60)},{6*sin(60)});
%         \coordinate[label={[blue] above:{$s_2$}}] (D) at 
%                 ({6*cos(120)},{6*sin(120)});
%         \coordinate[label={[yelloworange] left:{$s_3$}}] (E) at (-5,2);
%         \coordinate (F) at ({4*cos(240)},{4*sin(240)});
%         \coordinate (G) at ({4*cos(-60)},{4*sin(-60)});

%         \path[name path=line60] (A) -- (C);
%         \path[name path=line120] (A) -- (D);
%         \path[name path=line0] (B) -- (E);
%         \path[name intersections={of=line60 and line0, by=J}];
%         \path[name intersections={of=line120 and line0, by=K}];

%         \draw[draw opacity=0, fill=gray!30] (A) -- (J) -- (K) -- (A);
        
%         \centerarc[red,thick,<->]({4.8*cos(60)}, {4.8*sin(60)})(10:110:0.73);
%         \centerarc[blue,thick,<->]({4.8*cos(120)}, {4.8*sin(120)})(70:170:0.73);
%         \centerarc[yelloworange,thick,<->]({4.8*cos(180)+1}, {4.8*sin(180)+2})(130:230:0.73);        

%         \draw (A) edge[thick, draw=red,->] (C);
%         \draw (A) edge[thick, draw=blue,->] (D);
%         \draw (A) edge[thick, draw=red,->] (F);
%         \draw (A) edge[thick, draw=blue,->] (G);
%         \draw (B) edge[thick, draw=yelloworange,<->] (E);

% \end{tikzpicture}
% \end{center}
%     \caption{hyperplanes $s_1,s_2,s_3$ in $\tilde{A}_2$}
%     \label{fig:s123}
% \end{figure}

\begin{figure}[h!]
    \centering
    \includegraphics[width=0.4\linewidth]{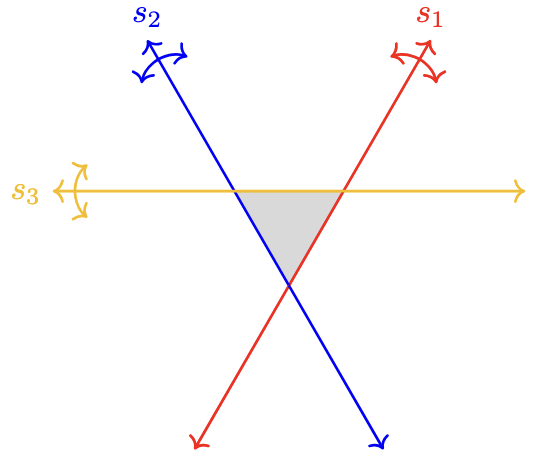}
    \caption{hyperplanes $s_1,s_2,s_3$ in $\tilde{A}_2$}
    \label{fig:s123}
\end{figure}

\begin{figure}[h!]
    \centering
    \includegraphics[width=0.7\linewidth]{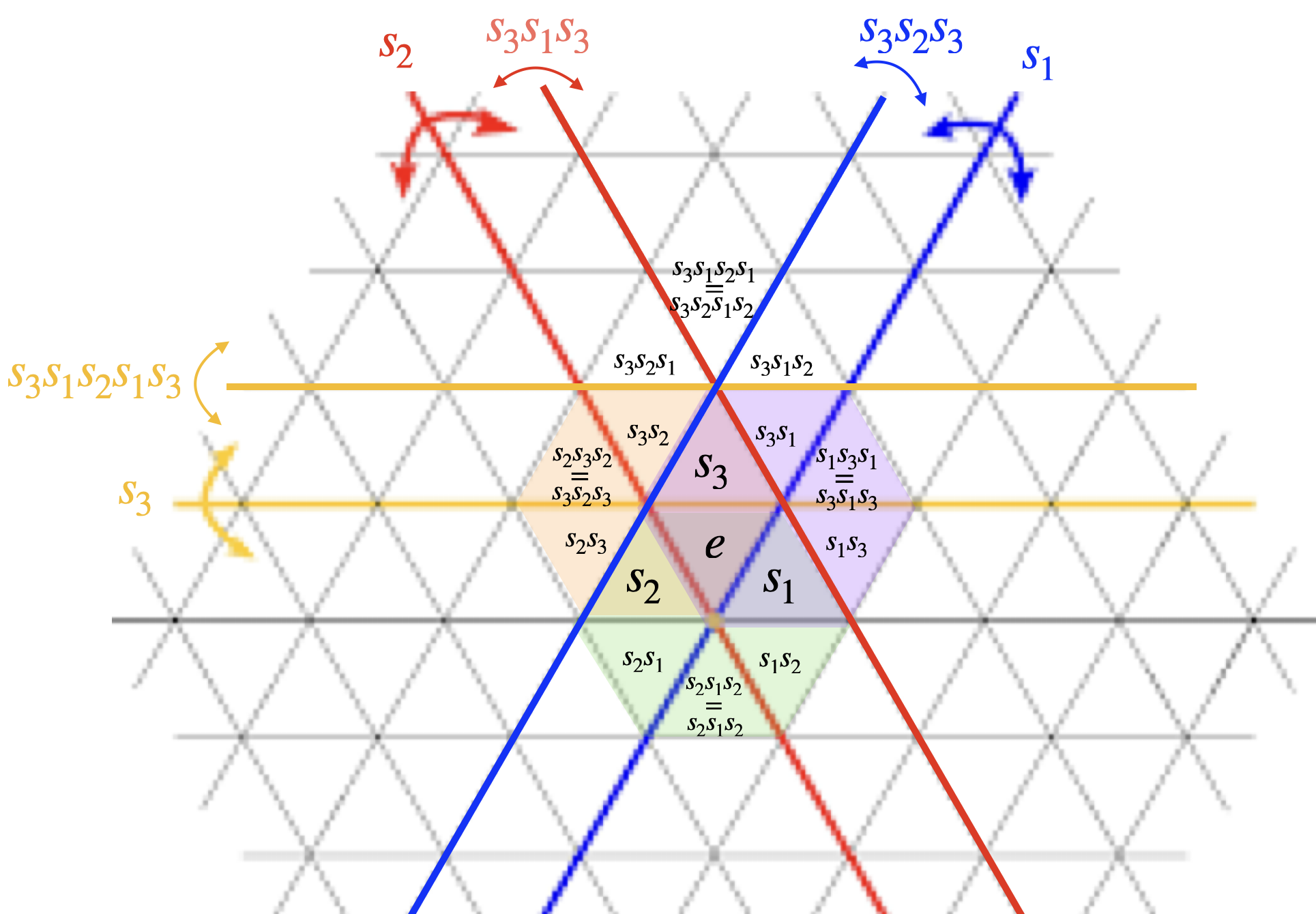}
    \caption{Tessellating Coxeter Complexes}
    \label{fig:tessCoxComp}
\end{figure}

\begin{defn} \label{defn: building} [Definition 4.1, \cite{abramenko2008buildings}; Definition 6.1, \cite{thomas2018geometric}] A \textit{building} is a simplicial complex made up of the union of subcomplexes (called \textit{apartments}) that satisfy the following axioms:
    \begin{enumerate}
        \item Every apartment is a Coxeter complex (defined below).
        \item For any two simplices, there is an apartment that contains both of        them.
        \item For any two apartments $A$, $A'$, there is an isomorphism $A\to A'$
                that fixes the intersection pointwise.
    \end{enumerate}
\end{defn}

Just as the link of a vertex in the $\tilde{A}_2$ Coxeter complex is an $A_2$ Coxeter complex, we have that the link of a vertex in the $\tilde{A}_2$ building is an $A_2$ building.  (See Figure \ref{fig:link}.) Note with respect to the star and link picture, every edge of a triangle has two more triangles coming off of it.  This is because the residue field is $\F_2$.  For residue field $\F_q$, there would be $q$ triangles coming off of each edge, which means that $q+1$ triangles share an edge.  Here is one example of an $\tilde{A}_2$ building with residue field $F_2$ \cite{2011.11707} (see Figure \ref{fig:tildeA2}):

\begin{figure}[h!]
    \centering
    \includegraphics[width=0.7\linewidth]{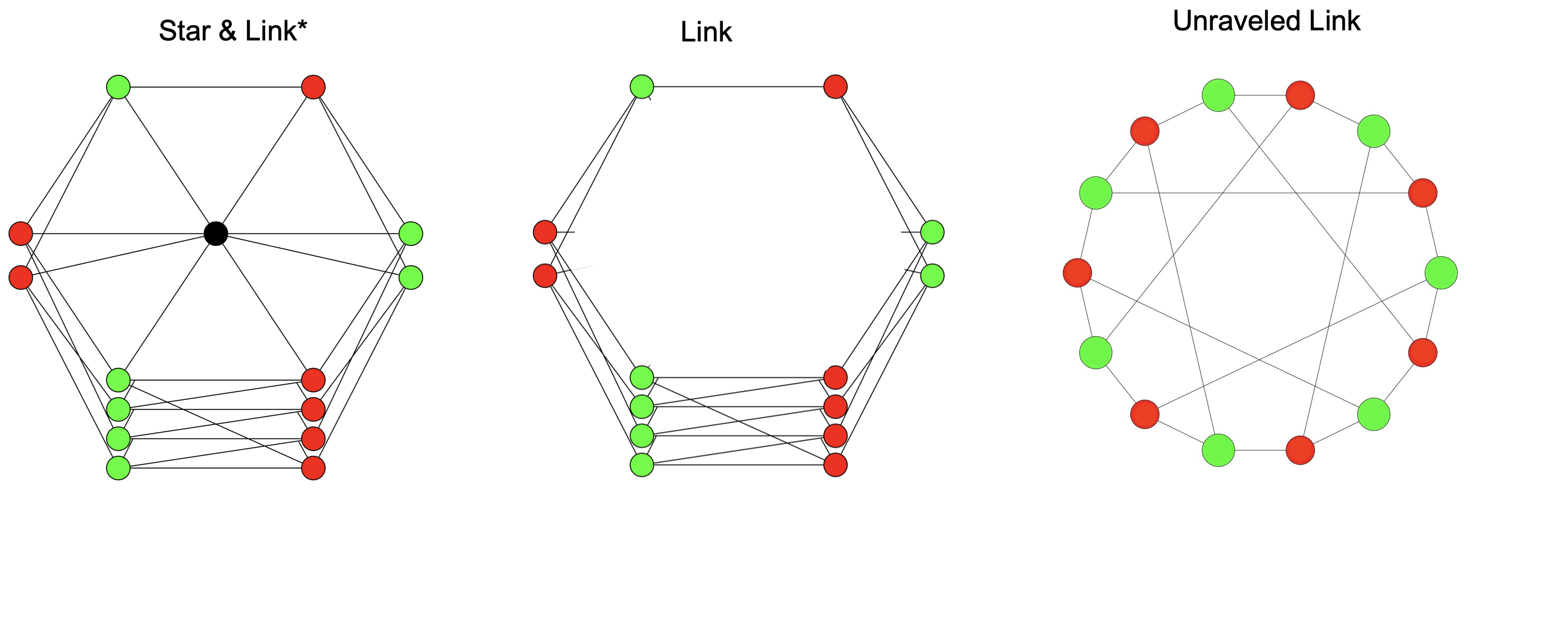}
    \vspace*{-7mm}
    \caption{Link of vertex in $\tilde{A}_2$ with residue field $\F_2$; image taken from \cite{Robertson}}
    \label{fig:link}
\end{figure}

\begin{figure}[h!]
    \centering
    \includegraphics[width=0.5\linewidth]{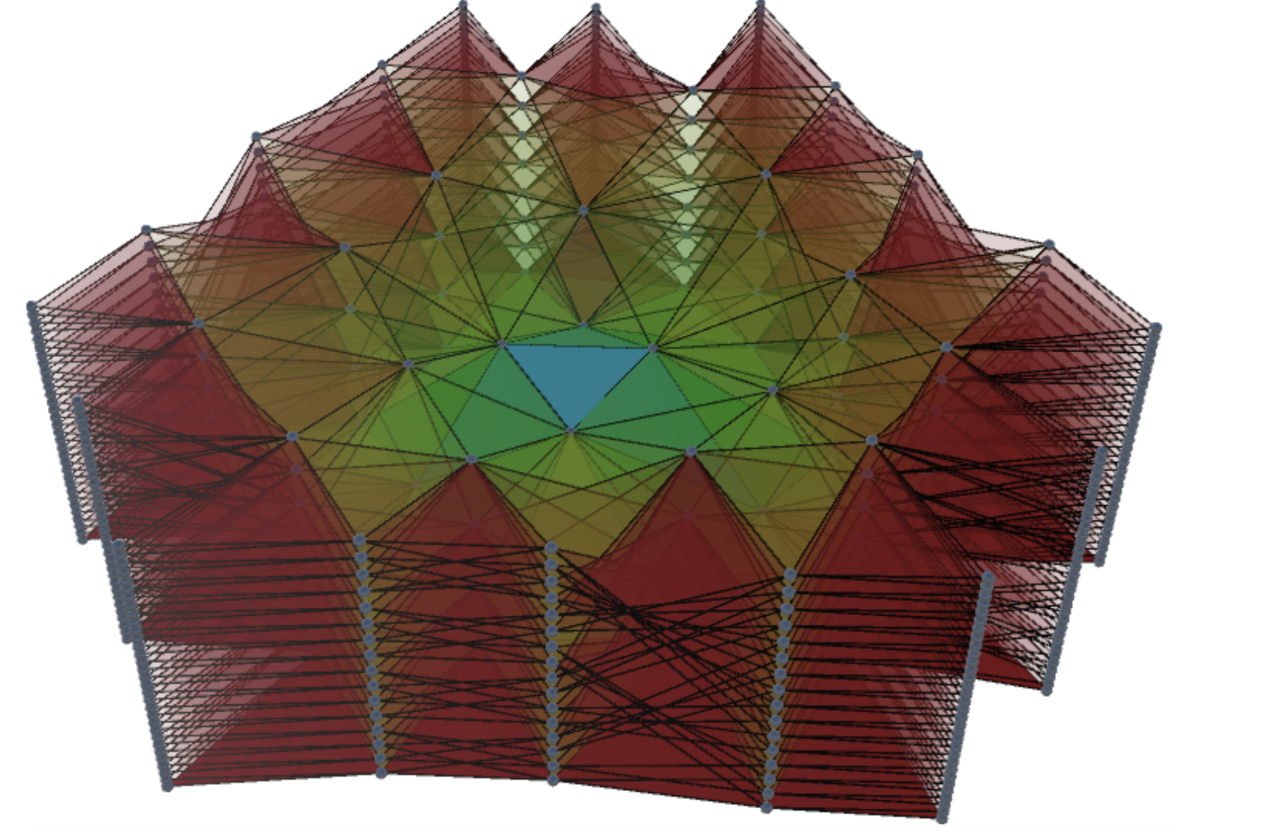}
    \caption{An $\tilde{A}_2$ building with residue field $\F_2$; image taken with permission from \cite{2011.11707}}
    \label{fig:tildeA2}
\end{figure}

%\hfill\\
%\vspace*{40mm}
%\newpage
\subsection{Triangle Presentations} \label{subsection:tripres}\hfill\\
\par This section covers the definition of a triangle presentation and what it means for two triangle presentations to be equivalent.  It also provides a concrete example of a triangle presentation.  Triangle presentations $\mathscr{T}$ are a key way to distinguish $\tilde{A}_2$ buildings from each other.  They become the relations in an abstract group $\Gamma_\mathscr{T}$ whose Cayley graph is the 1-skeleton of the building.  

\begin{defn} \label{defn: tripres} [Definition \S 3, \cite{cartwright1993groups}]
    Let $P$ and $L$ be the sets of points and lines respectively in a (finite or infinite) projective plane. A bijection $\alpha: P \to L$ is called a \textit{point-line correspondence}. A set $\mathscr{T}$ of triples $(a,b,c)$, where $a,b,c \in P$, is called a \textit{triangle presentation over $P$ compatible with $\alpha$} (or just \textit{triangle presentation} for short) if the following hold:
    \begin{enumerate}
        \item given $a,b\in P$, then $(a,b,c)\in \mathscr{T}$ for some $c\in P$ if and only if $b$ and $\alpha(a)$ are incident;
        \item $(a,b,c)\in\mathscr{T}$ implies that $(b,c,a)\in \mathscr{T}$;
        \item given $a,b \in P$, then $(a,b,c)\in\mathscr{T}$ for at most one $c\in P$.
    \end{enumerate}
\end{defn}

Let $\Gamma_\mathscr{T}$ denote the group with presentation $\langle a_i, \text{ for all } i\in P \,: \, a_ia_ja_k=1$ for all $(i,j,k)\in \mathscr{T} \rangle$. We can understand the relations of $\Gamma_\mathscr{T}$ by observing in Definition \ref{defn: tripres} that (1) enforces all relations to form triangles in the Cayley graph using the given generating set, (2) is the group property of conjugation, and (3) avoids setting distinct generators equal to each other.

Using Definition \ref{defn: tripres} parts (1) and (2), we can read triangle presentation element $(a,b,c)$ to mean that $b\in \alpha(a)$, $c\in \alpha(b)$, and $a\in \alpha(c)$. 
Note that Definition \ref{defn: tripres} lists elements $(a,b,c)$, $(b,c,a)$, and $(c,a,b)$ separately.  But exactly the same group is obtained by including only one of each cyclically permuted elements as relations.  This gives a group presentation with strictly fewer relations. We will heretofore refer to $\mathscr{T}$ as the cyclic permutation equivalence class of $\mathscr{T}_{\text{\tiny{CMSZ}}}$ from \cite{cartwright1993groups}.
%This more concise group has a strictly smaller minimal set of relations. 

The definition of a triangle presentation  can best be understood with an example (labeled $A.2$ in \cite{cartwright1993groupsb}).  (See Figure \ref{fig:TriPresF2}.) The points are labeled 0 through 6 and lines are three element subsets in keeping with the projective plane axioms:

\begin{figure}[h!]
    \centering
    $
    \begin{array}{ccccc}
        \alpha: P\to L & & & \mathscr{T} \\
        \alpha (0)=\{1,2,4\} & & & (3,3,1) & 3\in \alpha(3), 1\in \alpha(3), 
                3\in \alpha(1)\\
        \alpha (1)=\{3,4,6\} & & & (6,6,2) \\
        \alpha (2)=\{5,6,1\} & & & (5,5,4) \\
        \alpha (3)=\{0,1,3\} & & & (1,4,2) \\
        \alpha (4)=\{2,3,5\} & & & (0,1,6) \\
        \alpha (5)=\{4,5,0\} & & & (0,2,5) \\
        \alpha (6)=\{6,0,2\} & & & (0,4,3) \\
    \end{array}
    $
    \caption{Triangle presentation for a projective plane of order 2}
    \label{fig:TriPresF2}
\end{figure}

The following observation will be needed for the proof of Proposition \ref{prop:2^k}.

\begin{obs} \label{obs:tripres}
    Given a compatible $\alpha$, we can deduce the triangle presentation.  Worst case scenario, this can be done via brute force using the definition of a triangle presentation.  And given a triangle presentation, we can construct $\alpha$. This can be done in a straight-forward manner.  We know that each triangle presentation element, say $(a,b,c)$, tells us on which lines points $a$, $b$, and $c$ are.  If $a=b=c$, then we know that point $a$ is on line $\alpha(a)$.  Otherwise, we know on which lines three points are.  Thus, a triangle presentation fully encodes the map $\alpha:P\to L$.  

    Via the example shown in Figure \ref{fig:TriPresF2}, we can see how the triangle presentation encodes $\alpha$: 
        \begin{figure}[h!]
        \begin{center}
        \begin{tabular}{ c c c }
            ${\color{red}3}\in \alpha(3), {\color{red}1}\in \alpha(3), 
                {\color{red}3}\in \alpha(1)$ & & \hspace*{5mm}
                $\alpha(0)=\{{\color{Purple}1}, {\color{ForestGreen}2}, {\color{Bittersweet}4}\}$\\
            ${\color{blue}6}\in \alpha(6), {\color{blue}2}\in \alpha(6), 
                {\color{blue}6}\in \alpha(2)$ & & \hspace*{5mm}
                $\alpha(1)=\{{\color{red}3}, {\color{cyan}4}, {\color{Purple}6}\}$\\
            ${\color{orange}5}\in \alpha(5), {\color{orange}4}\in \alpha(5), 
                {\color{orange}5}\in \alpha(4)$ & & \hspace*{5mm}
                $\alpha(2)=\{{\color{blue}6}, {\color{cyan}1}, {\color{ForestGreen}5}\}$\\    
            ${\color{cyan}4}\in \alpha(1), {\color{cyan}2}\in \alpha(4), 
                {\color{cyan}1}\in \alpha(2)$ & & \hspace*{5mm}
                $\alpha(3)=\{{\color{red}3}, {\color{red}1}, {\color{Bittersweet}0}\}$\\
            ${\color{Purple}1}\in \alpha(0), {\color{Purple}6}\in \alpha(1), 
                {\color{Purple}0}\in \alpha(6)$ & & \hspace*{5mm}
                $\alpha(4)=\{{\color{orange}5}, {\color{cyan}2}, {\color{Bittersweet}3}\}$\\
            ${\color{ForestGreen}2}\in \alpha(0), 
                {\color{ForestGreen}5}\in \alpha(2), 
                {\color{ForestGreen}0}\in \alpha(5)$ & & \hspace*{5mm}
                $\alpha(5)=\{{\color{orange}5}, {\color{orange}4}, {\color{ForestGreen}0}\}$\\
            ${\color{Bittersweet}4}\in \alpha(0), 
                {\color{Bittersweet}3}\in \alpha(4), 
                {\color{Bittersweet}0}\in \alpha(3)$ & & \hspace*{5mm}
                $\alpha(6)=\{{\color{blue}6}, {\color{blue}2}, {\color{Purple}0}\}$
        \end{tabular}
        \end{center}
            \caption{Triangle Presentation encodes $\alpha$}
            \label{fig:TriPres encodes alpha}
        \end{figure}
\end{obs}

% \begin{defn}
%     If the associated abstract groups of two triangle presentations are isomorphic, then we say that the two triangle presentations are \textit{equivalent}. 
% \end{defn}

% \begin{rmk}
%     The \cite{cartwright1993groupsb} paper defines equivalency of triangle presentations based solely on the triangle presentations themselves (Definition \S 2).  The lemma that immediately follows their definition (Lemma 2.2) gives the definition that we're using here for simplicity.  
% \end{rmk}

The following lemma, definition, and proposition show how to determine whether two triangle presentations are equivalent.  This concept will be heavily used in Section \ref{sec: equiv tri pres}.

\begin{lem}[Lemma 2.1, \cite{cartwright1993groupsb}] \label{lem: collineation and correlation}
    Let $\mathscr{T}$ be a triangle presentation compatible with point-line correspondence $\alpha: P\to L$.
    \begin{enumerate}
        \item Let $h$ be a collineation of $(P,L)$.  Then $h(\mathscr{T})=\{\, 
            \big(h(x),h(y),h(z)\big)  :  (x,y,z)\in \mathscr{T}\}$ is a triangle presentation compatible with the point-line correspondence $h\alpha h^{-1}:P\to L$.
        \item Let $\mathscr{C}$ be a correlation of $(P,L)$.  Then $\mathscr{C}\alpha(\mathscr{T}^{\text{rev}})= \{\, \big( \mathscr{C}\alpha(z), \mathscr{C}\alpha(y), \mathscr{C}\alpha(x) \big) : (x,y,z)\in \mathscr{T}\}$ is a triangle presentation compatible with the point-line correspondence $\mathscr{C}\alpha^{-1}\mathscr{C}^{-1}: P\to L$.  
    \end{enumerate}
\end{lem}

\begin{defn}[page 170, \cite{cartwright1993groupsb}]
    Let $(P,L)$ be a projective plane with a correlation $\mathscr{C}$.  Let $\alpha,\alpha':P\to L$ be two point-line correspondences that admit triangle presentations $\mathscr{T}$ and $\mathscr{T}'$, respectively.  We say that $\mathscr{T}$ and $\mathscr{T}'$ are \textit{equivalent} if $\mathscr{T}'=h(\mathscr{T})$ or $\mathscr{T}'=h\mathscr{C}\alpha(\mathscr{T}^{\text{rev}})$ for some (possibly trivial) collineation $h$.
\end{defn}

\begin{prop}[Lemma 2.2, \cite{cartwright1993groupsb}] \label{prop:equiv tri pres means isom gp}
    If triangle presentations $\mathscr{T}$ and $\mathscr{T}'$ are equivalent, the associated abstract groups $\Gamma_\mathscr{T}$ and $\Gamma_{\mathscr{T}'}$ are isomorphic.
\end{prop}

Some maps $\alpha$ admit more than one triangle presentation as will be shown in Proposition \ref{prop:2^k}.  But every triangle presentation is associated with a unique map $\alpha$.  Further, very few maps admit a triangle presentation.  One can see that the complexity of finding via brute force the maps that do admit a triangle presentation grows factorially with respect to the order of the projective plane:  a projective plane of order $q$ has $(q^2+q+1)$! bijective maps.

The main theorems of \cite{cartwright1993groupsb} elucidate this complexity:

\begin{thm} [Theorem 1, \cite{coordinatizing}]
    For $q=2$, there are eight inequivalent triangle presentations whereby four $\Gamma_\mathscr{T}$'s embed into $\PGL\big(3,\F_2\llrrparen{t} \big)$ and another four embed into $\PGL(3,\Q_2)$.
\end{thm}

\begin{thm} [Theorem 2, \cite{cartwright1993groupsb}]
    For $q=3$, there are 89 inequivalent triangle presentations whereby 16 $\Gamma_\mathscr{T}$'s embed into $\PGL\big(3,\F_3\llrrparen{t} \big)$, another 8 embed into $\PGL(3,\Q_3)$, and the remaining 65 (called exotic $\tilde{A}_2$ buildings) do not embed in any $\PGL(3,K)$ for any local field $K$ with residue field $\F_3$.
\end{thm}

\hfill\\

\subsection{Perfect Difference Sets} \label{subsection:PDS} \hfill\\

\par This section covers definitions pertaining to perfect difference sets, shows how to create new perfect difference sets from existing \PDS s, and introduces \textit{multipliers} and theorems concerning them that are needed to prove the main result of this paper.

\begin{defn}
A $(v,k,\lambda)$-\textit{difference set} is a type of block design (see for example 
    \cite{dinitz2007handbook} or \cite{baumert2006cyclic}) consisting of a triple $(v,k,\lambda)$ that satisfy the following:
    \begin{enumerate}
        \item a group $G$ has order $v$,
        \item a subset $D$ of $G$ is called the \textit{difference set} and has size $k$,       and
        \item the multiset $\{d_id_j^{-1}: d_i, d_j\in D\}$ contains every non-identity element of $G$ exactly $\lambda$ times.
    \end{enumerate}
\end{defn}

In 1938, Singer \cite{singer1938theorem} connected difference sets to projective geometry by proving that difference sets can be constructed as follows:  
    \begin{enumerate}
        \item $v$ is the number of codimension 1 vector subspaces of 
                $\F_q^{n+1}$, where $q$ is a prime power, and   
                $G=\Z/{\left(\scriptscriptstyle \frac{q^{n+1}-1}{q-1}\right)}\Z$
        \item $k$ is the number of codimension 2 vector subspaces contained in each 
                codimension 1 subspace, and 
                $D=\{ i\in G : 
                \Tr_{\F_{q^{n+1}}^\times/\F_q^\times}(\zeta^i)=0 \}$, where  
                $\Tr$ is the trace function and $\zeta$ is a primitive element of $\F_{q^{n+1}}^\times/\F_q^\times$.  (Recall that $\F_q^{n+1}$ and $\F_{q^{n+1}}$ are isomorphic as $\F_q$-vector spaces.)
        \item $\lambda$ is the number of common codimension 2 vector subspaces in the intersection of any pair of codimension 1 subspaces \\
    \end{enumerate}

Note that when $n=2$, these numbers describe a projective plane:
    \begin{enumerate}
        \item $v=\frac{q^3-1}{q-1}=q^2+q+1$ (the total number of points/lines)
        \item $k=\frac{q^2-1}{q-1} = q+1$ (the number of points on a line)
        \item $\lambda=\frac{q-1}{q-1} = 1$ (the number of times two line intersect) \\
    \end{enumerate}
When $\lambda =1$, the difference set is called \textit{perfect}. 

In 1963, Halberstam and Laxton \cite{halberstam1963perfect} discovered a new method to find the subset $D$: 
 given a primitive element $\zeta$ for $\mathbb{F}_{q^3}$ over $\F_q$, write elements of $\mathbb{F}_{q^3}$ with basis $\{1, \zeta, \zeta^2\}$ over $\F_q$, and take the span of 
 $\begin{bsmallmatrix}
 1\\0\\0
 \end{bsmallmatrix}$ 
 and 
 $\begin{bsmallmatrix}
 0\\1\\0
 \end{bsmallmatrix}$.
 The exponents of the primitive elements corresponding to the elements in the span form a subset $D$ of $\mathbb{Z}/v\mathbb{Z}$, such that $D$ is a perfect difference set.

\begin{exmp} \label{exmp:PDS q=2}
Here is an example of a perfect difference set corresponding to $q=2$:
$(|\mathbb{Z}/7\mathbb{Z}|, |\{0,1,3\}|, 1)$.  Taking differences mod 7, we get the following:
$$\begin{matrix}
    0-1=6 & & 1-0=1 & & 3-0=3 \\
    0-3=4 & & 1-3=5 & & 3-1=2
\end{matrix}$$
\end{exmp}

For the rest of this paper, we will only be concerned with perfect difference sets where our group $G=\mathbb{Z}/v\mathbb{Z}$, $v=q^2+q+1$, and $q$ is a prime power.
Because our group is $\mathbb{Z}/v\mathbb{Z}$, we will use additive notation to refer to the differences of two elements in the subset $D$.  
Moreover, $\Z/v\Z$ is also a ring where our second binary operation is multiplication. 
Heretofore, we will refer to subset $D$ as the ``perfect difference set" unless an ambiguity arises.
\\
\subsubsection{Constructing Perfect Difference Sets from Perfect Difference Sets} \hfill\\
\par There are two ways to construct perfect difference sets from other perfect difference sets.  

\begin{defn}
    A \textit{shift} of a \PDS\ is an element of $\mathbb{Z}/v\mathbb{Z}$ added to the perfect difference set mod $v$.  
\end{defn}

For instance, from Example \ref{exmp:PDS q=2} where $D=\{0,1,3\}$, we have that $2+\{0,1,3\}=\{2,3,5\}$ is still a perfect difference set:  Given $d_1,d_2$ in an initial perfect difference set such that $d_1-d_2=x\in \mathbb{Z}/v\mathbb{Z}\setminus 
\{0\}$ and a ``shift" of $y\in \mathbb{Z}/v\mathbb{Z}$, we still have that $(y+d_1)-(y+d_2)=x$. 

The second way is to multiply the perfect difference set by a unit of $\mathbb{Z}/v\mathbb{Z}$.  The automorphisms of $\mathbb{Z}/v\mathbb{Z}$ correspond precisely to multiplication mod $v$ by an integer coprime to $v$.  Thus the image of a perfect difference set under an automorphism is again a perfect difference set. For instance, from Example \ref{exmp:PDS q=2}, we have that 
$3\times\{2,3,5\}=\{6,2,1\}$.  Given $d_1,d_2$ in an initial perfect difference set such that $d_1-d_2=x\in \mathbb{Z}/v\mathbb{Z}\setminus \{0\}$ and an integer $y$ that is relatively prime to $v$, we have $yd_1-yd_2=y(d_1-d_2)=yx.$\\
\\
\subsubsection{Multipliers}\label{sec: mutipliers} \hfill\\

\par This section defines the term \textit{multiplier} and includes the necessary background culminating in Theorem \ref{thm: 1 or 3 shifts fixed}. 

\begin{defn} 
    A \textit{multiplier}  is a \underline{nontrivial} automorphism of a perfect difference set's underlying group $G=\mathbb{Z}/v\mathbb{Z}$ that takes the \PDS\ to a shift of itself.  
\end{defn}

For example in $\mathbb{Z}/7\mathbb{Z}$, when we multiply the perfect difference set $\{2,3,5\}$ by 3, we get the \PDS\ $\{6,2,1\}$.  But 3 is not a multiplier of $\{2,3,5\}$ because there is no shift that takes $\{2,3,5\}$ to $\{6,2,1\}$.  However, $2\times\{2,3,5\}=\{4,6,3\}=1+ \{2,3,5\}$.  Thus 2 is a multiplier of $\{2,3,5\}$. 

We now recall some results concerning multipliers. 

\begin{thm}
    [Theorem 3.31, \cite{stinson2008combinatorial}]\label{lem: gcd(m,v)=1} If $m$ is a multiplier of a \PDS\ of order $v$, then $\gcd(m,v)=1$.
\end{thm}

\begin{comment}
\begin{proof}
    Given a \PDS\ $D$, there must be two elements of $D$, say $d_1, d_2$, such that $d_1-d_2=1$. Say that $m$ is a multiplier of $D$, then $mD=x+D$ for some $x$ in the group containing $D$. So we have that $1=d_1-d_2=(x+d_1)-(x+d_2)=md_i-md_j=m(d_i-d_j)$ for some $d_i,d_j\in D$.  Thus $\gcd(m,v)=1$ for any multiplier of a \PDS.    
\end{proof}
\end{comment}

\begin{thm}
    [Theorem 3.1, \cite{hall1951cyclic} or Theorem 3.1, \cite{baumert2006cyclic}] 
    \underline{First Multiplier} \underline{Theorem}:  
    Let $p$ be a prime divisor of $k-\lambda$ such that $\gcd(p,v)=1$ and $p>\lambda$.
    Then $p$ is a multiplier of a \PDS\ for group $\Z/v\Z$.
\end{thm}

\begin{cor} \label{cor:p multiplier of PDS}
When $q=p^n$ ($n\in \Z^+$), $v=q^2+q+1$, $k=q+1$, and $\lambda=1$, we have that $p>\lambda$ and $\gcd(p,v)=1$.  Thus $p$ will always be a multiplier of a \PDS\ for group $\Z/v\Z$.
\end{cor}

\begin{prop}
    [\S 3, \cite{hall1951cyclic}]
    \label{prop: multiplier cyclic gp}
    If $m$ is a multiplier, so are all nontrivial elements of the cyclic group of automorphisms generated by $m$. 
\end{prop}

\begin{comment}
\begin{proof}
    Let m be a multiplier of \PDS\ $D$.  Then $mD=x+D$ for some $x$ in the group containing $D$.  Then $m^2D=m(x+D)=mx+mD=mx+x+D$ where $mx+x$ is a shift of $D$.  Continuing this process, we see that $m^i$ is a multiplier for any $i\in \Z^+$.  Thus, all nontrivial elements of $\langle m \rangle$ are multipliers.   
\end{proof}
\end{comment}

In particular, Corollary \ref{cor:p multiplier of PDS} and Proposition \ref{prop: multiplier cyclic gp} together state that $p^n=q$ is always a multiplier of a \PDS\ for group $\Z/v\Z$.  This fact will be used in Section \ref{sec: main thm}.

\begin{defn}
    We say that a \PDS\ $D$ is \textit{fixed} by a multiplier $m$ if $D=mD$.
\end{defn}

\begin{thm} [Theorem 3.4\footnote{Details of Baumert's proof are fleshed out in Appendix A.2 of \cite{herron2025}}, \cite{baumert2006cyclic}] \label{thm: gcd(m-1,v)}
\label{lem: gcd shifts fixed by m}
    Say \PDS\ $D$ has a multiplier $m$.  Then there exists exactly $gcd(m-1,v)=g$ shifts fixed by $m$.  Moreover, if $D$ itself is fixed by $m$, then $m$ also fixes shifts $D+n(v/g)$ for $n=0,1,...,g-1$.
\end{thm}

\begin{comment}
\begin{proof}
    Let $A$ be as defined in corollary \ref{cor:A invertible}.  
    Note that the rows (lines) of the matrix will all be shifts of some base line (row 0) determined by a \PDS.  
    
    Say $m$ is a multiplier of a \PDS.  Then all shifts (i.e., lines) also have $m$ as a multiplier.  Because $m$ is an automorphism that sends lines to lines, $m$ determines two permutation matrices, say $P$ and $Q$ such that $PAQ=A$. Also, the trace of $P$ (i.e., $\Tr(P)$) represents the number of lines (i.e., shifts) fixed by $m$, and $\Tr(Q)$ represents the number of points fixed by $m$. Together with $PAQ=A$, we have the following where $A$ is a nonsingular matrix as per corollary \ref{cor:A invertible}: 
    $$\Tr(P)=\Tr(AQ^{-1}A^{-1})=\Tr((AQ^{-1})A^{-1})=\Tr(A^{-1}(AQ^{-1}))=
    \Tr(Q^{-1})=\Tr(Q)$$
    
    Thus, the number of shifts fixed by $m$ is the same as the number of points fixed by $m$. 
    Because $\Tr(Q)$ is the number of points fixed by $m$, we have $g=mg$ (mod $v$) for each fixed point $g$ in our underlying group. 
    But $g\equiv_v mg \iff 0\equiv_v (m-1)g$ and $0\equiv_v (m-1)g$ has $(m-1,v)$ solutions (i.e., number of fixed points).  This means that the number of fixed shifts is also $(m-1,v)$.      
\end{proof}
\end{comment}

\begin{thm}[Theorem 3.5, \cite{baumert2006cyclic}]\label{lem: permutes shifts}
    If $m_1,m_2$ are two multipliers of the same \PDS\ $D$, then $m_2$ permutes the shifts fixed by $m_1$. 
\end{thm}

\begin{comment}
\begin{proof}
    Say $x+D$ is a shift fixed by $m_1$.  Then $m_1(x+D) = x+D$.  
    Now $m_1(m_2(x+D))=m_2(m_1(x+D))=m_2(x+D)$.  
    Thus, $m_1$ also fixes $m_2(x+D)$.

\end{proof}
\end{comment}

The following theorem (as rendered by Baumert in [Theorem 4.1, \cite{baumert2006cyclic}]) combines work of Evans and Mann \cite{evans1951simple} with one of Mann as cited in \cite{hall1947cyclic}.

\begin{thm}\label{thm: 1 or 3 shifts fixed}    
    One or three shifts of a \PDS\ are fixed by \underline{all} multipliers.  
    In particular, for $q=0$ and $q=2$ (mod 3) or $q=1$ and $p=2$ (mod 3), there is a unique shift fixed by all multipliers.  Otherwise, $q=1$ and $p=1$ (mod 3) and there are three shifts fixed by all multipliers.
\end{thm}

\begin{exmp} \label{exmp:thm 1 or 3}
Here are some examples of the perfect difference sets with shifts fixed by all multipliers:
\begin{itemize}
    \item[a)] $q=2$ ($q=2$ (mod 3)):\\
    \PDS\ $\{1,2,4\}$ is fixed by all multipliers mod 7 (i.e., multipliers 2 and 4).

    \item[b)] $q=3$ ($q=0$ (mod 3)):\\
    \PDS\ $\{0,1,3,9\}$ is fixed by all multipliers mod 13 (i.e., multipliers 3 and 9).

    \item[c)] $q=4$ ($q=1$ (mod 3) and $p=2$ (mod 3)):

        \begin{itemize}
            \item[$i$)] \PDS\ $\{0,2,7,8,11\}$ is fixed only by multipliers mod 21 that are congruent to 1 mod 3 (i.e., multipliers 4 and 16).
            \item[$ii$)] \PDS\ $\{0,1,4,14,16\}$ is fixed only by multipliers mod 21 that are congruent to 1 mod 3 (i.e., multipliers 4 and 16).
            \item[$iii$)] \PDS\ $\{7,9,14,15,18\}$ is fixed by \textit{all} multipliers mod 21 (i.e., multipliers 2, 4, 8, and 16).
        \end{itemize}
    \item[d)] $q=7$ ($q=1$ (mod 3) and $p=1$ (mod 3)):\\
        Perfect difference sets 
        $\{0,11,19,20,24,26,36,54\}$, 
        $\{0,1,5,7,17,35,38,49\}$, and 
        $\{16,19,30,38,39,43,45,55\}$
        are fixed by all multipliers mod 57 (i.e., multipliers 7 and 49).

\end{itemize}

\end{exmp}

Note that the discussion heretofore in this section showed the existence of \PDS s fixed by multipliers, not how many such \PDS s exist for any given multiplier.  For example, when $q=2$, there are exactly two \PDS s that are fixed by the multiplier 2 where neither is a shift of the other.  They are $\{1,2,4\}$ and $3,5,6$.  Enumerating the \PDS s fixed by all multipliers for a given order $q$ comes into play later in this paper (see Theorem \ref{lem: col or cor}).

\hfill

\section{Main Theorem and its Consequences}\label{sec: main thm}

This section furnishes the main theorem of this paper---that \PDS s encode triangle presentations.  It also provides an efficient algorighm to construct $\mathscr{T}$ and shows that \PDS s of the same order encode equivalent triangle presentations.  Finally, 
we look at extensions of $\Gamma_\mathscr{T}$ by particular automorphisms of $\Gamma_\mathscr{T}$.  
\\
\subsection{Main Theorem} \label{subsection:main thm} \hfill\\

\par We begin this section by first proving that the \PDS\ of order $q$ constructed by Singer \cite{singer1938theorem} is fixed by multiplier $q$.  This provides easy computation of one perfect difference of order $q$ that is fixed by multiplier $q$ that can be used to derive others and can be used in the algorithm to construct triangle presentations. Next, we prove a lemma that will be integral in proving the main theorem, Theorem \ref{thm:main thm}, which demonstrates that \PDS s encode triangle presentations.  Algorithm \ref{algorithm} provides an efficient method of constructing $\mathscr{T}$, and 
Corollary \ref{prop:2^k} shows that \PDS s in fact encode multiple triangle presentations.  

\begin{prop} \label{prop:trace 0 PDS fixed by multiplier q}
    The perfect difference set $D=\{d\in \F_{q^3/q}: \Tr(\zeta^d)=0\}$, where $\zeta$ is a primitive element in $\F_{q^3/q}$, is fixed by multiplier q.  
\end{prop}
\begin{proof}
    Say $d\in D$.  Then $\Tr(\zeta^d)=\zeta^d+\zeta^{dq}+\zeta^{dq^2}=0$.  Now $\Tr(\zeta^{dq})=\zeta^{dq}+\zeta^{dq^2}+\zeta^{dq^3}$.  But $q^2+q+1=0 \iff q^3+q^2+q=0 \iff q^3-1=0 \iff q^3=1$.  Thus, $\zeta^{dq^3}=\zeta^{d}$, which means that $\Tr(\zeta^d)=\Tr(\zeta^{qd})$.  
\end{proof}

Recall that Corollary \ref{cor:p multiplier of PDS} and Theorem \ref{prop: multiplier cyclic gp} showed us that $p^n=q$ is always a multiplier of a \PDS\ for group $\Z/v\Z$. This fact results in the following lemma:

\begin{lem} \label{lem:orbits}
    The elements of all perfect difference sets that are fixed by multiplier $q$ have orbit $\{d_i, qd_i, q^2d_i\}$ where the orbit order is either 3 (a triple) or 1 (a fixed point). 
\end{lem}

\begin{proof}
    Let a \PDS\ $D$ contain $k$ points $\{d_1, ..., d_k\}$ and be fixed by multiplier $q$. Recall from the proof of Proposition \ref{prop:trace 0 PDS fixed by multiplier q} that $q^3=1$. Thus the orbit of $d_i$ is $\{d_i, qd_i, q^2d_i\}$. If $d_i=qd_i$, the orbit of $d_i$ has order 1.  If $d_i\neq qd_i$, then $q^2d_i\neq q^3d_i=d_i$ and $q^2d_i\neq qd_i$, so the orbit has order 3.
\end{proof}

% \begin{note}
%     We will denote a triple in a \PDS\ $D$ fixed by multiplier $q$ as $\langle d_i, qd_i, q^2d_i\rangle$ regardless of whether the orbit has order 1 or 3 for reasons that will become clear in the proof of Theorem \ref{thm:main thm}.  
% \end{note}

\begin{construction} \textit{Triangle presentations from perfect difference sets} 
\par Let $D=\{d_1, ..., d_k\}$ be a perfect difference set fixed by multiplier $q$ for the group $\Z/v\Z$ where $v=q^2+q+1$.  Let $\{0, 1, 2, ..., v-1\}$ be the points of the projective plane, and let $\{D, 1+D, 2+D, ..., v-1+D\}$ be the lines.  Define $\alpha: P\to L$ by $\alpha(i)=i+D$.  
We construct a triangle presentation corresponding to $\alpha$ using the orbits from Lemma \ref{lem:orbits} under the action of the multiplier $m=q$, or alternatively $m=q^2$, which is the inverse of $q$ in $\Z/v\Z$.  Both choices for $m$ share these critical features:  $mD=D$ and $1+m+m^2=0$ (mod $v$). 

Given $d_i\in D$ for $i\in \{1, ..., k\}$, let $\langle d_i, md_i, m^2d_i \rangle$ denote the following set of triples (to be used as relations) generated by the orbit of $d_i$ (which may be a single point, then $d_i=md_i=m^2d_i$):

$$\langle d_i, md_i, m^2d_i \rangle = \{(j, j+d_i, j+d_i+md_i) : j\in \Z/v\Z\}$$

\end{construction}

% Consider a \PDS\ $D=\{d_1,...,d_k\}$ for the group $\Z/v\Z$ where $v=q^2+q+1$ and $D$ is fixed by the multiplier $q$.  Recall that $j+D$ is a shift of $D$ by $j$.  The sets $\{j+D:j\in \Z/v\Z\}$ correspond to lines in a Desarguesian projective plane of order $q$ (with points corresponding to elements of $\Z/v\Z$.  Define $\alpha: P\to L$ by $\alpha(j)=j+D$.  
% We can now state and prove our main theorem result. 

\vspace*{0mm}
\begin{thm}\label{thm:main thm}
    Let $m=q$ or $q^2$.  
    The set
    $\mathscr{T}_{\text{\tiny{CMSZ}}}=\bigcup_{i=1}^k\langle d_i, md_i, m^2d_i \rangle$ is a triangle presentation compatible with $\alpha$ as defined above.  
\end{thm}

\begin{proof}

% Recall that if $(a,b,c)$ is an element of a triangle presentation 
% $\mathscr{T}$, then its cyclic permutation $(b,c,a)$ is as well.  So if we take the difference of consecutive points in $(a,b,c)$, we get $b-a$, $c-b$, and $a-c$.  
% Similarly, for $j\in \Z/v\Z$, the difference of consecutive points within $(j,j+d_i, j+d_i+qd_i)$ is $(j+d_i)-j=d_i$, 
% $(j+d_i+qd_i)-(j+d_i)=qd_i$, and 
% $j-(j+d_i+qd_i)=-d_i-qd_i=-d_i-qd_i -q^2d_i+q^2d_i=-d_i(1+q+q^2)+q^2d_i=q^2d_i$ (because $1+q+q^2=0$ (mod $v$)).  
% Combining these differences in order, we get the following:
% $$\langle \ (j+d_i)-j, \ (j+d_i+qd_i)-(j+d_i), \ j-(j+d_i+qd_i) \ \rangle
% =\langle d_i, qd_i, q^2d_i\rangle$$

% We claim that the union of all triples and fixed points in a \PDS\ $D$ with multiplier $q$ having shift 0 is a triangle presentation for a projective plane of order $q$. 
% Specifically, the triple or fixed point $\langle d_i, qd_i, q^2d_i\rangle$ in a \PDS\ $D$ having shift 0 represents the subset $\big\{(j,j+d_i, j+d_i+qd_i) : \ j\in \mathbb{Z}/v\mathbb{Z}\big\}$ of a (potential) triangle presentation $\mathscr{T}=\bigcup_{i=1}^k \big\{(j,j+d_i, j+d_i+qd_i) : \ j\in\mathbb{Z}/v\mathbb{Z}\big\}$ that corresponds to bijective map 
%     $\alpha:$ Points $\to$ Lines where $\alpha(j)=j+D$.  
    
We verify that the three conditions in the definition of a triangle presentation \ref{defn: tripres} are met.  
We will start by first checking condition 2 followed by conditions 1 and 3.

\hfill\\
\underline{Condition 2}:  Elements of $\mathscr{T}_{\text{\tiny{CMSZ}}}$ are closed under cyclic permutation.
\\
\par It suffices to show that $(j+d_i, j+d_i+md_i, j)\in \mathscr{T}_{\text{\tiny{CMSZ}}}$ for all $j$ and $d_i$.  Let $\ell=j+d_i$ and $d_n=md_i$.  
We need to show that $j=\ell+d_n+md_n$.  
We have that $\ell+d_n+md_n= j+d_i +md_i + m(md_i)=j+d_i(1+m+m^2)=j$ (mod $v$) as required.

\hfill\\
\underline{Condition 1}:  Given $a,b\in P$, then $(a,b,c)\in \mathscr{T}_{\text{\tiny{CMSZ}}}$ for some $c\in P$ if and only if $b$ and $\alpha(a)$ are incident.
\\
\par
Given $a,b\in P$, assume there exists a $c\in P$ such that $(a,b,c)\in \mathscr{T}_{\text{\tiny{CMSZ}}}$. Then $b=a+d_i$ for some $d_i\in D$ and $c=a+d_i+md_i$.  
Now $\alpha(a)=a+D$; thus, $b=a+d_i\in a+D=\alpha(a)$.  
Therefore, $b$ and $\alpha(a)$ are incident as claimed.

Given $a,b\in P$, assume that $\alpha(a)$ and $b$ are incident.  
Then $b\in \alpha(a)=a+D$, so $b=a+d_i$ for some $d_i\in D$.  
Setting $c=a+d_i+md_i$, it follows by definition that $(a,b,c)=(a,a+d_i,a+d_i+md_i)\in \mathscr{T}_{\text{\tiny{CMSZ}}}$.

\hfill\\
\underline{Condition 3}: Given $a,b \in P$, then $(a,b,c)\in\mathscr{T}_{\text{\tiny{CMSZ}}}$ for at most one $c\in P$. 
\\
\par The definition of $\mathscr{T}_{\text{\tiny{CMSZ}}}$ determines $c$ by the formula $c=m(b-a)$.
\end{proof}

% Note that if a \PDS\ is fixed by multiplier $q$, then it is also fixed by multiplier $q^2$.  

% \begin{cor}\label{cor:main thm}

%     A $(|\Z/v\Z|, |D|,1)$-difference set with $D=\{d_1,...,d_k\}$ a \PDS\ fixed by  multiplier $q^2$ encodes a triangle presentation 
%     $\mathscr{T}=$ \\$\bigcup_{i=1}^k\{\langle d_i, q^2d_i, qd_i \rangle : d_i\in D \}$ 
%     that corresponds to bijective map $\alpha:$ Points $\to$ Lines where $\alpha(j)=j+D$.
% \end{cor}

As a consequence of Theorem \ref{thm:main thm}, we see a new intrinsic connection between \PDS s and triangle presentations.  

\begin{algm} \label{algorithm}

Recall from Subsection \ref{subsection:tripres} that $\mathscr{T}$ is the cyclic permutation equivalence class of $\mathscr{T}_{\text{\tiny{CMSZ}}}$.  We now show how to efficiently construct the non-redundant $\mathscr{T}$ which equals $\mathscr{T}_{\text{\tiny{CMSZ}}}/\Z_3$ where $\Z_3=\langle s \rangle$ acts on $\mathscr{T}_{\text{\tiny{CMSZ}}}$ by $(a,b,c)\xmapsto{s} (b,c,a)$.  
Recall that $k=|D|$.  Let $f$ be the number of fixed points of the action of $q$ on $D$.  Then $t=\frac{1}{3}(k-f)$ is the number of order 3 orbits of this action.  The proof of condition 2 above shows that $\mathscr{T} = 
\bigcup_{d\in D'} \langle d,md,m^2d \rangle$ where $D'\subsetneq D$ is a set of orbit representatives in $D$.  Note that $|D'|=f+t$.  Thus, we see that $\mathscr{T}$, a set of size $\mathcal{O}(q^3)$, is compressed into the data of its corresponding perfect difference set, which has size $\mathcal{O}(q)$.

In particular, we bin the $k$ elements of a \PDS\ fixed by $q$ into triples as follows:  Pick an element of the \PDS, multiply it by $m$ and then $m^2$, remove this triple (possibly singleton) from the remaining elements of the \PDS\ and pick a remaining element of the \PDS\ to start this process over.  The final step is to take a triple, say $\langle a,b,c\rangle$ and compute $(b-a +i, c-b+i, a-c+i)$ for $i=0,...,q^2+q$.  Thus we see that the \PDS\ of size $\mathcal{O}(q)$ comprises all the data for $\mathscr{T} = 
\bigcup_{d\in D'} \langle d,md,m^2d \rangle$. 

\end{algm}

The below corollary follows immediately from the proof of Theorem \ref{thm:main thm} and the discussion above by observing that the proof also works when letting $m$ vary between $q$ and $q^2$ for each orbit.  But we provide another proof from the point of view that we can construct $\alpha$ given a triangle presentation (Observation \ref{obs:tripres}).

\begin{cor} \label{prop:2^k}
    Re-index $D'$ as $D'=\{d'_1, ..., d'_{f+t}\}$.  Then 
    there are at least $2^t$ distinct triangle presentations compatible with $\alpha$ of the form $\mathscr{T}=\bigcup_{i-1}^{f+t} \langle d'_i, m_id'_i, m_i^2d'_i \rangle$  where $m_i=q$ or $q^2$.  
\end{cor}

\begin{proof}

    Let $\{ d,qd,q^2d \}$ be  an order-3 orbit of $D$.  
    It is sufficient to show that the triangle presentation elements denoted by $\langle d,qd,q^2d \rangle$ and $\langle d,q^2d,qd \rangle$ enumerate the same points on a line.  
    Consider the triangle presentation element $(0,d,d+qd)$ that is associated with $\langle d,qd,q^2d \rangle$.  This means that $\alpha(0)$ contains $d$, $\alpha(d)$ contains $d+qd$, and $\alpha(d+qd)$ contains $0$.  And one triangle presentation element associated with $\langle d,q^2d,qd \rangle$ is $(0,d,d+q^2d)$, which means that $\alpha(0)$ also contains $d$.  Another triangle presentation element associated with $\langle d,q^2d,qd \rangle$ is $(d+qd,\ 2d+qd,\ 2d+qd+q^2d)$.  But $2d+qd+q^2d=d + (1+q+q^2)d$ where $q+q+q^2=0$.  Thus $2d+qd+q^2d=d$, which makes triangle presentation element $(d+qd,\ 2d+qd, \ 2d+qd+q^2d)=(d+qd,\ 2d+qd,\ d)$.  And this shows us that $\alpha(d)$ contains $d+qd$.  
    Lastly, $(qd, \ d+qd, \ d+qd+q^2d)$ is also a triangle presentation element that is associated with $\langle d,q^2d,qd \rangle$.  Again because $d+qd+q^2d=0$, triangle presentation element $(qd, \ d+qd, \ d+qd+q^2d) = (qd, \ d+qd, 0)$, which means that $\alpha(d+qd)$ contains 0. 

    We can of course reverse this process to show that the triangle presentation elements associated with $\langle d,q^2d,qd \rangle$ enumerate the same elements on a line as those associated with $\langle d,qd,q^2d \rangle$. 

    Thus we see that for any \PDS\ fixed by multiplier $q$ or $q^2$, there are $2^t$ corresponding distinct triangle presentations because every order-3 orbit under $q$ of $D$ can be represented as either $\langle d,qd,q^2d \rangle$ or $\langle d,q^2d,qd \rangle$.  
\end{proof}

The above proposition means that we can ``mix and match" triples $\langle d,qd,q^2d \rangle$ and $\langle d,q^2d,qd \rangle$ within \PDS.  For example, the \PDS\ $\{1,2,4\}$ of order $2$ fixed by the multiplier $2$ corresponds to two distinct triangle presentations $\langle 1,2,4 \rangle$ and $\langle 1,4,2 \rangle$ but both are compatible with the same bijective map $\alpha$.  In fact, we can now see that each \PDS\ of order $q$ and fixed by multiplier $q$ corresponds to exactly two triangle presentations for orders 2, 3, and 4.  (See Section \ref{sec:examples}.)  

But once we get to $q=5$, there are four such triangle presentations:  Consider the \PDS\ $\{1,5,17,22,23,25\}$ of order 5 that is fixed by the multiplier 5.  This \PDS\ decomposes into two distinct orbits of order 3 where each orbit can be represented by $\langle d,qd,q^2d \rangle$ or $\langle d,q^2d,qd \rangle$.  The first distinct orbit can be represented by $\langle 1,5,25 \rangle$ or $\langle 1,25,5 \rangle$, and the second distinct orbit can be represented by $\langle 17,23,22 \rangle$ or $\langle 17,22,23 \rangle$.  Therefore the four triangle presentations compatible with $\alpha(0) = \{1,5,17,22,23,25\}$ are $\{\langle 1,5,25 \rangle, \langle 17,23,22 \rangle\}$, $\{\langle 1,5,25 \rangle, \langle 17,22,23 \rangle\}$, $\{\langle 1,25,5 \rangle, \langle  17,23,22\rangle\}$, and $\{\langle 1,25,5 \rangle, \langle 17,22,23 \rangle\}$. 
When mixing ``ordered orbits" fixed by $q$ and $q^2$, one gets a group $\Gamma$ whose Cayley graph is the 1-skeleton of an exotic building as found by Alex Lou\'e \cite{loue2024infinite} using the ideas presented in this paper.

\vspace{5mm}

\subsection{Equivalent Triangle Presentations}\label{sec: equiv tri pres}
\hfill\\
\par In this section, we show that all \PDS s of the same order encode equivalent triangle presentations as defined by \cite{cartwright1993groupsb}.  We begin by noting that shifts of a \PDS\ of order $q$ correspond to collineations of Desarguesian projective planes of order $q$ and automorphisms of a \PDS\ of order $q$ correspond to correlations of a projective plane of order $q$.

The following is a restatement of a direct consequence of Theorem 2 in \cite{halberstam1964perfect} as mentioned at the beginning of \S 4 in \cite{halberstam1964perfect}. 

\begin{thm} \label{lem: col or cor}
    Given two \PDS s of order $q$, we can transform one to the other via a series of shifts or automorphisms.
\end{thm}

\begin{thm} \label{thm: PDS collineation or correlation}
    The projective planes described by any two \PDS s of order $q$ that are also fixed by multiplier $q$ can be transformed to one another via a collineation or correlation.
\end{thm}

\begin{proof} 
    $ $\newline
    \underline{\textit{Case 1}}: If both \PDS s generate the same set of lines, there exists a collineation between them.
    
    \par Assume there are two \PDS s, say $D$ and $D'$, of order $q$ that are also fixed by multiplier $q$ and have the same set of lines.  Let $\alpha(0)=D$ and $\alpha'(0)=D'$, where $\alpha$ and $\alpha'$ are compatible with the corresponding triangle presentations associated with $D$ and $D'$, respectively.  
    Per Theorem \ref{thm: 1 or 3 shifts fixed}, we know that there exist perfect difference sets for all orders $q$ that are fixed by the multiplier 
    $q$. And by Theorem \ref{thm: gcd(m-1,v)}, for a \PDS\ fixed by multiplier $q$, there are $gcd(q-1,v)$ shifts of that \PDS\ that are also fixed by $q$.  Using the division algorithm, we note that $v=q^2+q+1=(q+2)(q-1)+3$, which next means that we need to consider the $gcd(q-1,3)$.  Thus, if $q=0$ or $2$ (mod $3$), there is exactly one shift of a \PDS\ fixed by $q$.  
    
    Say $q=1$ (mod 3).  Let $D$ be a \PDS\ fixed by $q$.  Then there are two more shifts of $D$ that are also fixed by $q$.  By Lemma \ref{thm: gcd(m-1,v)}, these shifts are $D+1(v/g)$ and $D+2(v/g)$. Under the $\alpha$ map, we have  $\alpha(j)=j+D$.  But we also have $\alpha(j)=j+ D+1(v/g)$ and $\alpha(j)=j+ D+2(v/g)$, which are just collineations that send line $j+D$ to lines $j+ D+1(v/g)$ and $j+ D+2(v/g)$, respectively.  Note then that $D'$ is just a shift of $D$.\\
    \\
    \underline{\textit{Case 2}}: If both \PDS s generate a different set of lines, there exists a correlation between them.
    
    \par Assume that the two \PDS s of order $q$ that are also fixed by multiplier $q$ do not have the same set of lines.  Then by \cite{halberstam1964perfect}, the second \PDS\ must be a nontrivial automorphism of the first \PDS.  That is, the second \PDS\ must be a multiple, say m, of of the first \PDS\ where $gcd(v,m)=1$.  We claim that this automorphism corresponds to a correlation of a projective plane. 
    
    %Because multiplying by any number relatively prime to the modulus v is an automorphism of $\Z/v\Z$, multiplying a perfect difference set that is fixed by multiplier $q$ by a number relatively prime to $v$ yields a new perfect difference set that is also fixed by $q$. We claim that this automorphism corresponds to a correlation of a projective plane. 

    Let $D$ be a \PDS\ fixed by multiplier $q$.  Then the lines of the corresponding projective plane, say $\pi$, are $j+D$, $j=0,1,...,v-1$.  Let $c$ be a number relatively prime to $v$.  Then the lines of of the projective plane, say $\pi'$, corresponding to \PDS\ $cD$ are $j+cD$.  Let $P_\pi$ and $L_\pi$ refer to the points and lines, respectively, of $\pi$, and let $P_{\pi'}$ and $L_{\pi'}$ refer to the points and lines, respectively, of $\pi'$.  
    
    For $p\in P_\pi$, define a map $\mathscr{C}: (P_\pi\cup L_\pi)\to (P_{\pi'}\cup L_{\pi'})$ by 
    $\mathscr{C}(p)=cD -cp$ and $\mathscr{C}(D+p)=-cp$.   Note that $\mathscr{C}$ is a bijection that takes points to lines and lines to points.  Say point $p$ is on line $D+k$ with respect to projective plane $\pi$.  Then we have the following: 
        \begin{tabbing}
            \hspace*{10mm} $p\in (D+k) $ \= $\iff (p-k)=c^{-1}(cp-ck)\in D $ \\
                \> $\iff (cp-ck)\in cD $ \\
                \> $\iff -ck\in (cD-cp)$ \\
                \> $\iff \mathscr{C}(D+k)\in \mathscr{C}(p)$
        \end{tabbing}
    Because $\mathscr{C}$ is a bijection, we have the reverse direction as well. 
    Therefore, $\mathscr{C}$ is the correlation of a projective plane that corresponds to the automorphism ${c: \Z/v\Z\to \Z/v\Z}$ where $c(z)=cz$ for $z\in \Z/v\Z$.     
\end{proof}

Note that when applying Lemma \ref{lem: collineation and correlation} part (2) to obtain the new triangle presentation $\mathscr{C}\alpha(\mathscr{T}^{\text{rev}})$, sometimes it yields the triangle presentation in Theorem \ref{thm:main thm} with $m=q$ and sometimes it yields the triangle presentation with $m=q^2$.

\begin{rmk} 
    All triangle presentations as constructed in Theorem \ref{thm:main thm} for a given $q$ embed as an arithmetic subgroup of $\PGL\big(3,\F_q\llrrparen{t}\big)$.
    By Theorem \ref{thm: PDS collineation or correlation}, we see that all the triangle presentations as constructed in Theorem \ref{thm:main thm} for a given $q$ are equivalent.  The construction of the triangle presentations, say $\mathscr{T}'$, in \S 4 of \cite{cartwright1993groups} uses the \PDS\ $D=\{d\in \F_{q^3/q}: \Tr(\zeta^d)=0\}$ where $\zeta$ is a primitive element in $\F_{q^3/q}$. They show that the abstract group, say $\Gamma_0$, with these triangle presentation relations embeds as an arithmetic subgroup of $\PGL\big(3,\F_q\llrrparen{t}\big)$. This means that all of the abstract groups associated with the triangle presentations in Theorem \ref{thm:main thm} also embed as arithmetic subgroups of $\PGL\big(3,\F_q\llrrparen{t}\big)$ because they are isomorphic to $\Gamma_0$ by Proposition \ref{prop:equiv tri pres means isom gp}.  
\end{rmk}

\vspace*{0mm}

\hfill\\

\subsection{Extension of Groups with Triangle Presentation Relations} \label{subsection:extension} \hfill\\

We now look at extending $\Gamma_\mathscr{T}$ by automorphisms of $\Gamma_\mathscr{T}$ that are induced by automorphisms of $\mathscr{T}$.  

Article \cite{cartwright1993groupsb} defines $\tilde \Gamma_{\mathscr{T}}$ as the extension of $\Gamma_{\mathscr{T}}$ by a subgroup of Aut$(\mathscr{T})$, where Aut$(\mathscr{T})=\{h\in \text{ collineations} : h(\mathscr{T})=\mathscr{T}\}$.  In our case, all the projective planes are Desarguesian, so the collineation group of points is known to be $P\Gamma L(3,\F_q)$.  Moreover, Theorem 4.1 of \cite{cartwright1993groups} tells us that this extension also embeds as an arithmetic subgroup of $\PGL\big(3,\F_q\llrrparen{t}\big)$.

This subgroup is generated by the two collineations of a projective plane described below.

\begin{itemize}
    \item[1.]  The first group of collineations is generated by a permutation $p$ associated to the $\alpha$ map that corresponds to the triangle presentation.  That is, $\alpha(i)=L_{p(i)}$ for $L_{p(i)}$ the $p(i)^{\text{th}}$-labeled line.  With respect to our triangle presentations, $p$ can always be defined as $p(i)=i+1$, whereby the order of $p$ is $v$.  
    \item[2.] The second group of collineation is generated by the Frobenius automorphism $s$ of the field extension $\F_{q^3}/\F_q$ that, in additive notation, sends point $i$ to $qi$.  Then the order of $s$ is always 3.
\end{itemize}

% For projective planes orders 2 and 3, \cite{cartwright1993groupsb} show that Aut$(\mathscr{T})$ is generated by two cyclic collineation groups.  However, we can readily extend this for any prime power $q$.  The two groups are as follows:  

It is particularly east to see the action of $p$ and $s$ on $\mathscr{T}$ when $\mathscr{T}$ is described by \PDS s.
Recall by Theorem \ref{thm:main thm} that for $d_i$ in a \PDS, $\langle d_j, qd_j, q^2d_j \rangle = \{(i, i+d_j, i+d_j+qd_j):  i\in P\}\subseteq \mathscr{T}$. Thus we see that permutation $p\in \text{Aut}(\mathscr{T})$.  
Moreover, for $(i, i+d_j, i+d_j+qd_j)\in \langle d_j, qd_j, q^2d_j \rangle$, $q(i,\ i+d_j, \ i+d_j+qd_j)$ is such that the cyclic differences are $qd_j$, $q^2d_j$, and $d_j$.  This means that $q(i,\ i+d_j, \ i+d_j+qd_j)\in \langle qd_j, q^2d_j, d_j \rangle = \langle d_j, qd_j, q^2d_j \rangle$.  Similarly, $\langle q^2d_j, d_j, qd_j \rangle = \langle d_j, qd_j, q^2d_j \rangle$.  Consequently, $s\in \text{Aut}(\mathscr{T})$.  

We can see that the presentation of $\langle p,s \rangle$ is $\langle s,p : s^3=p^v=1, sps^{-1}=p^q \rangle$ and that its order is $3v$.  
% Further, each $h\in \text{Aut}(\mathscr{T})$ induces an automorphism of the group $\Gamma_{\mathscr{T}}$.
% Further, each $h\in \text{Aut}(\mathscr{T})$ acts by conjugation on the vertex set of points in the projective plane.  This vertex set is the generators of $\Gamma_{\mathscr{T}}$, so $ha_ih^{-1}=a_{h(i)}$.    
Thus, $\tilde \Gamma_{\mathscr{T}} = \Gamma_{\mathscr{T}} \rtimes \langle p,s \rangle = \langle a_0, ..., a_{v-1}, s, p : a_ia_ja_k=1 \text{ for all } (i,j,k)\in \mathscr{T}, \ s^3=p^v=1, \ sps^{-1}=p^q, \ pa_ip^{-1}=a_{p(i)}, \ sa_is^{-1}=a_{s(i)}\rangle$.

\hfill\\

\section{Connection with Panel-Regular Lattices} \label{subsection:extension} 

We now connect our automorphism groups to the type-preserving automorphism groups of Essert \cite{essert2013geometric} and Witzel \cite{witzel2017panel}.  Their groups 
act simply transitively (i.e., regularly) on each type of \textit{panel} of the building.  These groups are known as \textit{panel-regular lattices}.   Essert and Witzel also made use of Singer's construction of projective planes and connected this construction to \PDS s.   
We first give an overview of panel-regular lattices. 
Then we connect specific panel-regular lattices to our $\Gamma_\mathscr{T}$ via a common subgroup.  Lastly, we translate this subgroup using the nomenclature of triangle presentations, which creates a uniform understanding of the panel-regular groups and $\Gamma_{\mathscr{T}}$. \\

\subsection{Overview of Panel-Regular Lattices} \hfill\\

We learn in \cite{capdeboscq2015cocompact} that the subgroup of permutations, $\langle p \rangle$, in Aut$(\mathscr{T})$ is a special subgroup of $\PGL (3,q)$ called a Singer group.

\begin{defn} (\S 2, \cite{witzel2017panel})
        A \textit{Singer group} of $\PGL (3,q)$ is a subgroup of $\PGL (3,q)$ that acts simply transitively on the set of points (and lines) in a projective plane.  When the Singer group is cyclic, the generator of the Singer group is called a \textit{Singer cycle}.  
        
        % Per {reference at school}, if the projective plane is finite, then the Singer cycle also acts simply transitively on the set of lines 
\end{defn}

We next concern ourselves with two special types of lattices and one more group:

\begin{defn} (\S 1, \cite{witzel2017panel})
    A \textit{Singer lattice} is a lattice that preserves types and acts simply transitively on the three sets of panels (i.e., edges) of a given type of an $\tilde A_2$ building.  A Singer lattice is \textit{cyclic} if every vertex stabilizer is cyclic.
\end{defn}

% \begin{defn} (\S 1, \cite{witzel2017panel})
%     A \textit{Singer cyclic lattice} is a Singer lattice such that every vertex stabilizer of each vertex type is cyclic.
% \end{defn}

\begin{defn} (\S 4, \cite{capdeboscq2015cocompact})
    For vertices $v_0$, $v_1$, and $v_2$ of a standard maximal simplex of an $\tilde A_2$ building, let $S_i$ be the stabilizer of $v_i$ in $\Gamma_{\mathscr{T}} \rtimes \langle p \rangle$. 
    Define $\Gamma''$ to be the group generated by the three $S_i$'s: $\langle S_0, S_1, S_2 \rangle$. Note that each $S_i$ acts simply transitively on the set of neighboring vertices of a given type and that $\Gamma''$ is a subgroup of $\Gamma_{\mathscr{T}} \rtimes \langle p \rangle$.  Moreover, $\Gamma''$ is a Singer lattice.
\end{defn}

The first main theorem in \cite{capdeboscq2015cocompact} is as follows:
\begin{thm} \label{thm:gamma''}
    Let $\Delta$ be a building associated to $\PGL \big(3,\F_q\llrrparen{t}\big)$.      
    Then $\PGL \big(3,\F_q\llrrparen{t}\big)$ admits $\Gamma''$ as a cocompact lattice such that
        \begin{itemize}
            \item the action of $\Gamma''$ is type-preserving and transitive (not necessarily free) on each vertex type; and
            \item the stabilizer of each vertex in $\Gamma''$ is isomorphic to a cyclic Singer group in $\PGL(3,q)$ (whereby it follows that $\Gamma''$ acts simply transitively on the set of panels of each type in $\Delta$).
            % \item the stabilizer of each vertex in $\Gamma'$ is isomorphic to the normalizer of a Singer cycle in $\PGL(3,q)$.
        \end{itemize}
\end{thm}

\begin{defn}
    If $\sigma_0$, $\sigma_1$, and $\sigma_2$ are the generators of the vertex stabilizers of a chamber, triple $(\sigma_0, \sigma_1, \sigma_2)$ is called a \textit{chamber triple}.   
\end{defn}

For us, the Singer cycle is $\langle p \rangle$, which makes the vertex stabilizers cyclic.  Thus we see that $\Gamma''$ is a Singer cyclic lattice.  Let $\sigma_i$ be the generator of $S_i$ for $0\leq i\leq 2$.  Then $\Gamma''=\langle\sigma_0, \sigma_1, \sigma_2\rangle$.  The relations of $\Gamma''$ are as follows (Lemma 3.2 and Theorem 3.4, \cite{witzel2017panel}):
    \begin{itemize}
        \item $\sigma_0^{e_{i,0}} \sigma_1^{e_{i,1}} \sigma_2^{e_{i,2}} = 1$ for 
            $1\leq i\leq q+1$ where $(e_{i,j})$ is a $(q+1)\times 3$ matrix that contains the data for the exponents.  
        \item $\sigma_j^v=1$ for $1\leq j\leq 2$.  (Recall that $v=q^2+q+1$.)
    \end{itemize}

Section 5 of \cite{witzel2017panel} proves that there exists a unique Singer cyclic lattice that is an arithmetic subgroup of $\PGL\big(3,\F_q\llrrparen{t}\big)$ that acts on the same $\tilde A_2$ building as our $\Gamma_{\mathscr{T}}$ for a given $q$.  Call this unique lattice $\Gamma''_{\mathscr{A}}$.

A restatement of Corollary 5.2 in \cite{witzel2017panel} tells us exactly which difference matrix yields this arithmetic lattice.

\begin{thm}
    The group $\Gamma''_{\mathscr{A}}$ consists of relations with matrix $(e_{i,j})$ where $e_{i,0}=e_{i,1}=e_{i,2}$ for all $1\leq i\leq q+1$. Moreover, the columns of $(e_{i,j})$ are a \PDS.
\end{thm}

\hfill\\

\subsection{Subgroup common to $\Gamma_\mathscr{T}$ and $\Gamma''_{\mathscr{A}}$} \hfill\\

We are now ready to introduce the subgroup that is common to both $\Gamma_{\mathscr{T}}$ and $\Gamma''_{\mathscr{A}}$ and translate it into triangle presentation language. 

The beginning of section 7 in \cite{witzel2017panel} tells us that if each $\sigma_i$ maps non-trivially to the abelianization of $\Gamma''_{\mathscr{A}}$, then there exists a homomorphism $\phi: \Gamma''_{\mathscr{A}} \to \Z/v\Z$ such that $\sigma_i\mapsto 1$ for $0\leq i\leq 2$.  Thus, $\phi(g)$ is the word length with respect to $\{\sigma_0, \sigma_1, \sigma_2\}$ mod $v$ of $g$ for every $g\in \Gamma''_{\mathscr{A}}$.  
It also states that the kernel of $\phi$ acts freely on the vertices of each type.  Because $\Gamma''_{\mathscr{A}}$ acts transitively on each vertex type, the stabilizer of each vertex type has order $v$, and the index of $\ker \phi$ in $\Gamma''_{\mathscr{A}}$ is also $v$, we see that $\ker \phi$ is simply transitive on each vertex type.

% Article \cite{capdeboscq2015cocompact} then go on to prove in Lemmas 13 and 14 the following:
%     \begin{itemize}
%         \item If $3|(q+1)$, then $\Gamma'=\tilde \Gamma_{\mathscr{T}} \cap \text{PSL}\big(3, \F_q((x))\big)$; and
%         \item if $3|q$, then $\Gamma'' = \tilde \Gamma_{\mathscr{T}} \cap \text{PSL}\big(3, \F_q((x))\big)$
%     \end{itemize}

We are now in a position to create the poset of lattices as shown in Figure \ref{fig:poset}. Let the base be $\ker \phi$ from section 7 in \cite{witzel2017panel}. We can then extend $\ker \phi$ by a type-rotating automorphism, say $\langle r \rangle$, to recover $\Gamma_{\mathscr{T}}$, or we can extend $\ker \phi$ by the permutation automorphism (i.e., a Singer cycle) to get a Singer cyclic lattice.  Next, we can extend $\Gamma_{\mathscr{T}}$ by the permutation automorphism to get $\Gamma_{\mathscr{T}} \rtimes \langle p \rangle$, or we can extend the Singer lattice by $\langle r \rangle$ to get $\Gamma_{\mathscr{T}} \rtimes \langle p \rangle$.  Lastly, we can extend $\Gamma_{\mathscr{T}} \rtimes \langle p \rangle$ by the Frobenius automorphism $s$ to recover $\tilde \Gamma_{\mathscr{T}}$.

{
\begin{figure}[h!]
\begin{center}    
    \begin{tikzcd}[column sep=0em,row sep=2em]
        & \tilde \Gamma_{\mathscr{T}} = 
                (\Gamma_{\mathscr{T}} \rtimes \langle p \rangle) \rtimes \langle s\rangle 
           \\
        & \Gamma_{\mathscr{T}} \rtimes \langle p \rangle 
                \arrow[u, swap, "\text{extend by $\langle s \rangle$}"]\\
          \Gamma_{\mathscr{T}}
                \arrow[ur, "\text{extend by $\langle p \rangle$}"] 
                 && 
          \substack{\text{Singer lattice } \Gamma''_{\mathscr{A}} = \\ \ker \phi \rtimes \langle p \rangle 
            \footnotemark}
                \arrow[ul, swap, "\text{extend by $\langle r \rangle$}"] \\
        & \ker \phi 
                \arrow[ul,"\text{extend by $\langle r \rangle$}"] 
                \arrow[ur, swap, "\text{extend by $\langle p \rangle$}"] 
        
    \end{tikzcd}    
\end{center}
    \caption{Poset of Lattices}
    \label{fig:poset}
\end{figure}
}

This poset of lattices connects our lattices and the lattices in \cite{cartwright1993groups} and \cite{cartwright1993groupsb} to the Singer cyclic lattices in \cite{witzel2017panel} and \cite{essert2013geometric} and further shows how to restrict or extend one of these lattices to get to another.  In fact, we see that $\Gamma_{\mathscr{T}} = \ker \phi \rtimes \langle r \rangle$ and that $\Gamma_{\mathscr{T}}\rtimes \langle p \rangle = (\ker \phi \rtimes \langle p \rangle) \rtimes \langle r \rangle$.  This next part shows how to understand $\ker \phi$ via $\Gamma_{\mathscr{T}}$.  

If we consider an $\tilde A_2$ building colored by vertex types, we see that for a group's action to be type-preserving on vertices, each type must ``move via rhombuses" where the short diagonal of the rhombus is an edge with the same type as its opposite vertices.  We will make this more precise in the discussion to follow.  See the picture below, which is the restriction of the building to an apartment.  

\footnotetext{Observation 7.1 in \cite{witzel2017panel}}

\begin{figure}[h!]
\begin{center}    

\begin{tikzpicture}[scale=0.5]
    \coordinate (Origin)   at (0,0);
    \coordinate (XAxisMin) at (-3,0);
    \coordinate (XAxisMax) at (5,0);
    \coordinate (YAxisMin) at (0,-2);
    \coordinate (YAxisMax) at (0,5);
    
    \draw [thin,-latex, opacity=0 ] (XAxisMin) -- (XAxisMax) node[below] {};% Draw x axis
    \draw [thin,-latex, opacity=0 ] (YAxisMin) -- (YAxisMax) node[above] {};% Draw y axis
    
    \clip (-5.5,-5) rectangle (5.5,5); % Clips the picture...

    \begin{scope}[y=(60:1)]

        \foreach \x  [count=\j from 2] in {-8,-6,...,8}
        {        
          \draw[]
            (\x,-8) -- (\x,8)
            (-8,\x) -- (8,\x) 
            [rotate=60] (\x,-8) -- (\x,8) ;
        }
            \draw[ultra thick, fill=gray!20] (0,0) -- (2,0) -- (2,2) -- (0,2) -- cycle;
            \draw (0,2) -- (2,0);

            \draw[ultra thick, fill=gray!20] (0,0) -- (0,2) -- (-2,4) -- (-2,2) -- cycle;
            \draw (0,2) -- (-2,2);

            \draw[ultra thick, fill=gray!20] (0,0) -- (-2,2) -- (-4,2) -- (-2,0) -- cycle;
            \draw (-2,2) -- (-2,0);

            \draw[ultra thick, fill=gray!20] (0,0) -- (-2,0) -- (-2,-2) -- (0,-2) -- cycle;
            \draw (-2,0) -- (0,0-2);

            \draw[ultra thick, fill=gray!20] (0,0) -- (0,-2) -- (2,-4) -- (2,-2) -- cycle;
            \draw (0,-2) -- (2,-2);
           
            \draw[ultra thick, fill=gray!20] (0,0) -- (2,-2) -- (4,-2) -- (2,0) -- cycle;
            \draw (2,-2) -- (2,0);

                \foreach \x [count=\j from 2] in {-8,-6,...,8}
                    \foreach \y [count=\j from 2] in {-8,-6,...,8}
                {
                    \pgfmathsetmacro\result{\MyColorList[mod(\y+2*\x,3) < 0 ? mod(\y+2*\x,3) + 3 : mod(\y+2*\x,3)]}
                  
                    \node[circle, fill=\result, inner sep=2pt] at (\x,\y) {};
                    \node[circle, fill=\result, inner sep=4pt] at (0,0) {};
                }
    \end{scope}

\end{tikzpicture}
\end{center}
    \caption{Vertices colored by type. Center large red vertex has 6 possible moves one rhombus away to be type preserving.}
    \label{fig:rhombus}
\end{figure}

We can determine the group action for these rhombuses using triangle presentations.  
Because the Cayley grapy of $\Gamma_{\mathscr{T}}$ is the 1-skeleton of the $\tilde A_2$ building, we know that such rhombuses consisting of two triangles sharing an edge exist.  Say we have such a pairing, then there exist triangle presentation elements $(i,j,k)$ and $(j,\ell, m)$ that correspond to these two triangles with the respective group generators as shown in Figure \ref{fig:rhombus}:

\begin{figure}[h!]
\begin{center}   
\tikzset{
  A/.style={regular polygon, regular polygon sides=3,fill=gray!20,minimum size=3cm, draw},
  B/.style={regular polygon, regular polygon sides=3,fill=gray!20,minimum size=3cm, draw},
}
\begin{tikzpicture}[node distance=0pt, every node/.style={outer sep=0pt}]
  %\draw [help lines,gray!50,step =1] (-5,-5) grid (5,5);
  \node (A) [A=0] {};
  \node (B) [B=0, anchor=corner 2, rotate=180] at (A.corner 1) {};

    \begin{scope}[every node/.style={sloped,allow upside down}]

        \draw [draw=black]
            (A.corner 3) -- node {\midarrow} (A.corner 2) -- node {\midarrow} 
                (A.corner 1) -- node {\midarrow} cycle;
        \draw (B.corner 2) -- node {\midarrow} (B.corner 3) -- node {\midarrow} 
                (B.corner 1);
    \end{scope}
    
      \draw[shift=(A.side 2)] node[below] {$a_i$};
      \draw[shift=(A.side 1)] node[below right] {$a_j$};
      \draw[shift=(A.side 3)] node[above right] {$a_k$};
      \draw[shift=(B.side 2)] node[above] {$a_\ell$};
      \draw[shift=(B.side 3)] node[below left] {$a_m$};
    
      \node[circle, fill=red, inner sep=2pt] at (A.corner 3) {};
      \node[circle, fill=cyan, inner sep=2pt] at (A.corner 1) {};
      \node[circle, fill=blue, inner sep=2pt] at (A.corner 2) {};
      \node[circle, fill=red, inner sep=2pt] at (B.corner 3) {};

\end{tikzpicture}
\end{center}
    \caption{}
    \label{fig:rhombus}
\end{figure}

We will define rhombus elements $b_{i,m}=a_ia_m^{-1}$ where we always start with a generator of $\Gamma_{\mathscr{T}}$ and end with the inverse of a generator.  
Say $v$ and $w$ are both vertices of the same type.  We know that there exists an apartment that contains both vertices.  Thus, there exists a path from $v$ to $w$ via concatenation of rhombuses sides of type $a_ia_m^{-1}$ that begin and end at vertices of the same type.  Label this path by $\prod_{i=1}^n b_{i_n, j_n}$.

% We can translate all triangle presentation elements into rhombus elements by finding all pairs of elements that share a number, say $(i,j,k)$ and $(j,\ell,m)$ as above.  Then cyclically permute one triple to have the shared number in the middle and cyclically permute the other to have the shared number on the left.  Now we can think of the $j$ in $(j,\ell,m)$ as equal to $m^{-1}l^{-1}$, which enables us to substitute this new name for the $j$ in $(i,j,k)$.  This gives us ``triple" $(i, m^{-1}, \ell^{-1}, k)$, which we can split into two halves as $\big((i, m^{-1}), (\ell^{-1}, k)\big)$.  This duple preserves a vertex of the same type as the panel associated to group element $a_j$.  

%Without loss of generality, say the corresponding group elements take vertex $v_0$ to vertex $v'_0$.  If we had permuted our triples as $(j,k,i)$ and $(m,j,\ell)$ instead, we would get the corresponding duple $\big((m,i^{-1}), (k^{-1}\ell)\big)$.  This time the corresponding group elements would take $v'_0$ to $v_0$.  But the path traced by this duple is just the reversed path of duple $\big( (\ell^{-1}, k), (i, m^{-1}) \big)$, which is the cyclic permutation of $\big((i, m^{-1}), (\ell^{-1}, k)\big)$.  Thus we see that we are able to cyclically permute the duples just as we can cyclically permute the triples in the triangle presentation.  This leads to the following definition.

\begin{defn} \label{def:rhombus pres} 
    % Define a \textit{rhombus presentation} $\mathscr{R}({\mathscr{T}})$ to be the presentation corresponding to the duples created from the corresponding triangle presentation, and 
    Define $\Gamma_{\mathscr{R}(\mathscr{T})}$ to be the subgroup of $\Gamma_{\mathscr{T}}$ generated by all such $b_{i,m}$.  
\end{defn}

The following appears as a remark in Section 3 of \cite{cartwright1993groups}.  We take this occasion to record a proof in the literature.

\begin{lem} \label{lem:rhombus generators} 
    The group $\Gamma_{\mathscr{R}(\mathscr{T})}$, generated by $\{ b_{i,j}\}_{0\leq i<j \leq v-1}$ where $b_{i,j}= a_ia_j^{-1}$ in $\Gamma_{\mathscr{T}}$, corresponds to the type-preserving subgroup of $\Gamma_{\mathscr{T}}$.  
\end{lem}

\begin{proof}
     Let $w$ be a vertex of the $\tilde A_2$-building. Identify the 1-skeleton of the $\tilde A_2$-building with the Cayley graph of $\Gamma_{\mathscr{T}}$ so that $w$ may be considered an element of $\Gamma_{\mathscr{T}}$.  Then the link of $w$ is a bipartite graph of a projective plane where half of the the adjacent vertices $wa_i$ may be identified with points $i$ and the other half $wa_j^{-1}$ may be identified with lines $\alpha(j)$.  (Recall that $\alpha$ is the bijective map taking points to lines that corresponds to a triangle presentation.)  

    % Let $w$ be a vertex of the $\tilde A_2$-building.  Then the link of $w$ is a bipartite graph of a projective plane with half the vertices consisting of points $i$, $0\leq i\leq v-1$, and the other half consisting of lines $L_i=\alpha(i)$, where $\alpha$ is the bijective map compatible with $\mathscr{T}$.  Then there exists a group element $a_i$ that takes $w$ to point $i$ for every $i$.  Point $i$ can now be referred to as vertex $wa_i$.  By abuse of notation, call the edge between $w$ and vertex $wa_i$ edge $a_i$. 
    
    Then we note that the star of vertex $wa_i$ contains $v$ vertices of the same type as vertex $w$, one of which is $w$ itself.  Thus edge $a_i$ connects via vertex $wa_i$ to $q^2+q$ (or $v-1$) other edges that culminate in vertices of the same type as $w$.  This means that edge $a_i$ followed by any of the other $q^2+q$ edges, say $a_{j}^{-1}$, traces half of a rhombus with the short diagonal in the link of vertex $w$ and in the star of vertex $wa_i$.  If $i<j$, then include $b_{i,j}=a_{i}a_{j}^{-1}$ as a generator of $\Gamma_{\mathscr{R}(\mathscr{T})}$.  Otherwise, $(a_ia_{j}^{-1})^{-1}= a_{j}a_{i}^{-1}$ whereby we can relabel it $a_{i}a_{j}^{-1}$ so as to also include it as a generator. 
    %Thus, for every $b_{i,m^{-1}}$, there exists a $b_{\ell^{-1},k}$ such that ${(b_{i,m^{-1}})(b_{\ell^{-1},k})=1}$.
    See Figure \ref{fig:link and star} for an illustration of $q=2$:
\end{proof}

\begin{figure}[h!]
\begin{center}   
\tikzset{
  A/.style={regular polygon, regular polygon sides=3,fill=gray!40,
        fill opacity=0.5, minimum size=3cm, draw},
  B/.style={regular polygon, regular polygon sides=3,fill=gray!40,fill 
        opacity=0.5, minimum size=3cm, draw},
  C/.style={regular polygon, regular polygon sides=3,fill=gray!40,fill 
        opacity=0.5, minimum size=3cm, draw},
  D/.style={regular polygon, regular polygon sides=3,fill=gray!40,fill 
        opacity=0.5, minimum size=3cm, draw},
  E/.style={regular polygon, regular polygon sides=3,fill=gray!40,fill 
        opacity=0.5, minimum size=3cm, draw},
}
\begin{tikzpicture}[node distance=0pt, every node/.style={outer sep=0pt}]
  %\draw [help lines,gray!50,step =1] (-5,-5) grid (5,5);
  \node (A) [A=0] {};
  \node (B) [B=0, anchor=corner 2, rotate=180] at (A.corner 1) {};
  \node (C) [C=0, anchor=corner 1, rotate=60] at (A.corner 1) {};
  \node (D) [D=0, anchor=corner 1, rotate=240] at (A.corner 1) {};

   \node[circle, fill=cyan, inner sep=0pt] (a) at (-3.2,-0.7) {};
   \node[circle, fill=blue, inner sep=0pt] (b) at (3,0.9) {};
   \node[circle, fill=blue, inner sep=0pt] (c) at (-2.2,3.7) {};
   \node[circle, fill=blue, inner sep=0pt] (d) at (-4,0.5) {};
   \node[circle, fill=blue, inner sep=0pt] (e) at (-3.3,1.2) {};

    \begin{scope}[every node/.style={sloped,allow upside down}]
    
      \draw[gray, name path=line1] (A.corner 1) -- node {\midarrow}(a);
      \draw[name path=line2] (B.corner 1) -- (B.corner 3);

      \draw[gray, name path=line3] (b) -- node {\midarrow}(A.corner 1);
      \draw[name path=line4] (C.corner 1) -- (C.corner 2);

      \draw[gray, name path=line5] (c) -- node {\midarrow} (A.corner 1);
      \draw[name path=line6] (D.corner 3) -- (D.corner 2);

      \draw[gray, name path=line7] (d) -- node {\midarrow} (A.corner 1);
      \draw[name path=line8] (e) -- (a);
      
      \draw[gray, name path=line9] (e) -- node {\midarrow} (A.corner 1);
      \draw[name path=line10] (B.corner 3) -- (B.corner 1);
      
               %%%%%%%%%%%%%%%%%%%%%%%%%%%%%%%%%%%%%%%%%%%%
                
        \node[name intersections={of=line1 and line2, by=intersection_point}] at (intersection_point) {};

        \draw[draw=none, fill = gray!40, fill opacity=0.5] (A.corner 1) -- (intersection_point) 
            -- (a) -- (A.corner 2) -- (A.corner 1);
            
        \node[name intersections={of=line3 and line4, by=intersection_point2}] at (intersection_point2) {};

        \draw[draw=none, fill = gray!40, fill opacity=0.5] (C.corner 1) -- (intersection_point2) 
            -- (b) -- (C.corner 2) -- (C.corner 1);
            
         \node[name intersections={of=line5 and line6, by=intersection_point3}] at (intersection_point3) {};

        \draw[draw=none, fill = gray!40, fill opacity=0.5] (D.corner 3) -- 
        (c) -- (intersection_point3) -- (A.corner 1) -- (D.corner 3);
            
         \node[name intersections={of=line7 and line8, by=intersection_point4}] at (intersection_point4) {};
            
        \draw[draw=none, fill = gray!40, fill opacity=0.5] (A.corner 1) -- (-3.2,-0.7) -- (-4,0.5) -- (intersection_point4) -- (A.corner 1);

         \node[name intersections={of=line9 and line10, by=intersection_point5}] at (intersection_point5) {};
         
        \draw[draw=none, fill = gray!40, fill opacity=0.5] (A.corner 1) -- (-3.2,-0.7) -- (-3.3,1.2) -- (intersection_point5) -- (A.corner 1);

        % \draw[draw=red, fill =red, fill opacity=0.5] (A.corner 1) -- 
        % (a) -- (e) -- (intersection_point5) -- (A.corner 1);
            
              %%%%%%%%%%%%%%%%%%%%%%%%%%%%%%%%%%%%%%%%%%%%

         \draw[] (intersection_point) -- (a);
         \draw[] (a) -- node {\midarrow} (A.corner 2);

         \draw[] (intersection_point2) -- (b);
         % \draw[] (b) -- node {\midarrow} (C.corner 2);
         \draw[] (C.corner 2) -- node {\midarrow} (b);

         \draw[] (intersection_point3) -- (c);
         \draw[] (D.corner 3) -- node {\midarrow} (c);

         \draw[] (d) -- (intersection_point4);
         \draw[] (e) -- (intersection_point5);
         
         \draw[] (a) -- node {\midarrow} (d);
         \draw[] (a) -- node {\midarrow} (e);

               %%%%%%%%%%%%%%%%%%%%%%%%%%%%%%%%%%%%%%%%%%%%

    \end{scope}

    \begin{scope}[every node/.style={sloped,allow upside down}]

        \draw [draw=black]
            (A.corner 3) -- node {\midarrow} (A.corner 2) -- node {\midarrow} 
                (A.corner 1) -- node {\midarrow} cycle;
        \draw (B.corner 2) -- node {\midarrow} (B.corner 3) -- node     
                {\midarrow} (B.corner 1);
                
        \draw (C.corner 2) -- node {\midarrow} (C.corner 3) -- node {\midarrow}
            (A.corner 1);
        \draw (D.corner 3) -- node {\midarrow} (D.corner 2) -- node {\midarrow}
            (A.corner 1);
                
    \end{scope}
    
      \draw[shift=(A.side 2)] node[below] {};
      \draw[shift=(A.side 1)] node[below right] {$a_i$};
      \draw[shift=(A.side 3)] node[above right] {};
      \draw[shift=(B.side 2)] node[above] {};
      \draw[shift=(B.side 3)] node[below left] {};
    
      %\node[circle, fill=cyan, inner sep=0pt] (a) at (1,-1.3) {};
      \node[circle, fill=cyan, inner sep=2pt] (a) at (-3.2,-0.7) {};
      \node[circle, fill=red, inner sep=2pt] (b) at (3,0.9) {};
      \node[circle, fill=red, inner sep=2pt] (c) at (-2.2,3.7) {};
       \node[circle, fill=red, inner sep=2pt] (d) at (-4,0.5) {};
      \node[circle, fill=red, inner sep=2pt] (e) at (-3.3,1.2) {};
  
      \node[circle, fill=cyan, inner sep=2pt] at (A.corner 3) {};
      \node[circle, fill=blue, inner sep=2pt, label=above right:{$wa_i$}] at (A.corner 1) {}; %keep blue
      \node[circle, fill=red, inner sep=2pt, label=below:{$w$}] 
            at (A.corner 2) {};
      \node[circle, fill=cyan, inner sep=2pt] at (B.corner 3) {};
      \node[circle, fill=red, inner sep=2pt] at (C.corner 3) {};
      \node[circle, fill=red, inner sep=2pt] at (D.corner 2) {};
      % \node[circle, fill=green, inner sep=2pt] at (B.corner 3) {};
      % \node[circle, fill=orange, inner sep=2pt] at (B.corner 1) {};

\end{tikzpicture}
\end{center}
    \caption{}
    \label{fig:link and star}
\end{figure}

The following definition, proposition, and lemma will also be needed before we give a group presentation for $\Gamma_{\mathscr{R}(\mathscr{T})}$. 

\begin{defn}
    For an element $(i,j,k)$ of $\mathscr{T}$, define $(i,j,k)^{-1}$ to be $(k^{-1}, j^{-1}, i^{-1})$, which represents the reversed cycle. 
\end{defn}

% \begin{prop} \label{prop:apts are convex}
%     (Theorem 3.8, \cite{ronan2009lectures}; as stated in \cite{thomas2018geometric})
%     Apartments are convex, meaning that if $c_1$ and $c_2$ are any chambers in a building, and if $A$ is any apartment containing both $c_1$ and $c_2$, then $A$ contains every minimal gallery from $c_1$ to $c_2$. 
% \end{prop}

\begin{prop} \label{prop:apts are convex}
    (Theorem 3.8, \cite{ronan2009lectures})
    Apartments are convex, meaning that if $A$ is any apartment containing chamber $c$ and simplex $s$, then $A$ contains every minimal gallery from $c$ to $s$. 
\end{prop}

\begin{lem} \label{lem:min rhombus path}
    Let $A$ be an apartment containing chamber $c$ and vertex $v$ of type $t$. Then there exists a gallery $p$ through minimal rhombus moves from $c$ to $v$ that lies in $A$.  
\end{lem}

\begin{proof}
    Let $p=(c=c_0, c_1, ..., c_k)$ be a minimal gallery from vertex $w$ of type $t$ in $c$ to vertex $v$. If $(c_i,c_{i+1})$ crosses a panel of type $t$, then a rhombus move connects both vertices of type $t$.  If $(c_i, c_{i+1})$ crosses a panel not of type $t$, then $c_i$ and $c_{i+1}$ share a vertex of type $t$.  Thus, the sequence of type $t$ vertices $v_i$ in $c_i$ that starts at $w$ and ends at $v$ can be connected by rhombus moves between $v_i$ and $v_{i+1}$ whenever $v_i\neq v_{i+1}$.  
\end{proof}

In the below theorem, we will allow generators to be duplicated as 
$\{b_{i,j}\}_{\substack{i\neq j\\0\leq i,j \leq v-1}}$ in order to simplify the relations. 

\begin{thm} \label{prop:rhombus pres}
    $\Gamma_{\mathscr{R}(\mathscr{T})}$ has group presentation 
    $\langle b_{i,j} \text{ for } 0\leq i,j\leq v-1 \text{ and } {i\neq j} \, : \, b_{i,j}b_{j,i}=1 \text{ and } b_{i_2,j_2}b_{k_2,\ell_2}b_{m_2n_2}=1 \text{ for any six elements of $\mathscr{T}$ of the form } \\
    (i_1,i_2,i_3), \ 
    (i_3^{-1},j_2^{-1},j_3^{-1}),  \ 
    (j_3,k_2,k_3), \ 
    (k_3^{-1},\ell_2^{-1},\ell_3^{-1}), \ 
    (\ell_3,m_2,m_3), \text{ and } \\
    (m_3^{-1},n_2^{-1},i_1^{-1})
    \rangle$.  
\end{thm}

\begin{proof} Observe that these six elements of $\mathscr{T}$ are those whose corresponding elements in $\Gamma_{\mathscr{T}}$ form a hexagon in the Cayley graph of $\Gamma_{\mathscr{T}}$ as seen in Figure \ref{fig:hexagon}.
    \begin{figure}[h!]
        \centering
        \includegraphics[width=0.5\linewidth]{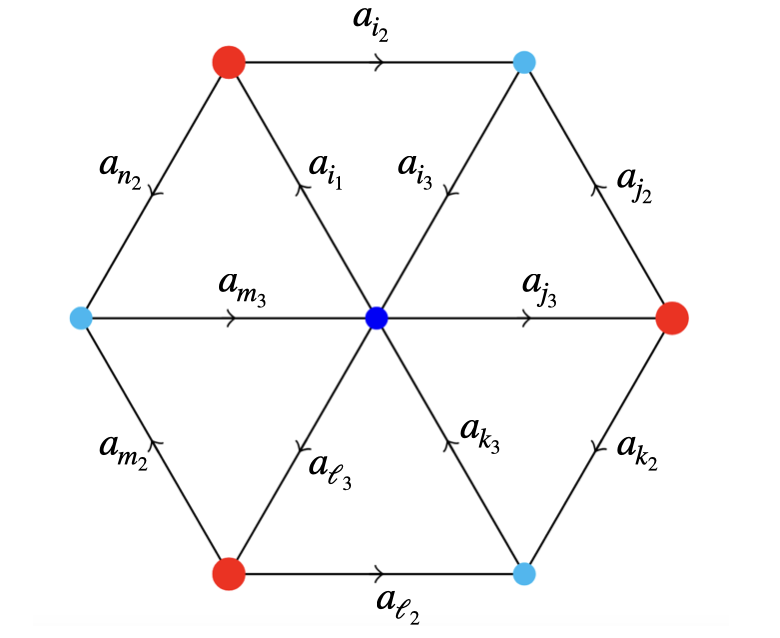}
        \caption{Hexagon in the Cayley graph of $\Gamma_{\mathscr{T}}$}
        \label{fig:hexagon}
    \end{figure}

    % \vspace*{-6mm}
    % \begin{tabbing}
    %      Then $1$ \= $=$ \small{$(a_{i_1}a_{i_2}a_{i_3})
    %     (a_{i_3}^{-1}a_{j_2}^{-1}a_{j_3}^{-1})
    %     (a_{j_3}a_{k_2}a_{k_3})
    %     (a_{k_3}^{-1}a_{\ell_2}^{-1}a_{\ell_3}^{-1})
    %     (a_{\ell_3}a_{m_2}a_{m_3})
    %     (a_{m_3}^{-1}a_{n_2}^{-1}a_{i_1}^{-1})$} \\
    %     \> $=a_{i_1}a_{i_2}a_{j_2}^{-1}a_{k_2}a_{\ell_2}^{-1}a_{m_2}
    %         a_{n_2}^{-1}a_{i_1}^{-1}$ \\
    %     \> $=a_{i_2}a_{j_2}^{-1}a_{k_2}a_{\ell_2}^{-1}a_{m_2}
    %         a_{n_2}^{-1} \ \ $ (by conjugating by $a_{i_1}^{-1}$) \\
    %     \> $=b_{i_2,j_2}b_{k_2,\ell_2}b_{m_2,n_2}$.
    % \end{tabbing}

Let $G$ be the group abstractly presented by $S=\{\beta_{i,j}\}$ and 
$R= \{\text{relations}\}$ where the first relator identifies $(\beta_{i,j})^{-1}$ with $\beta_{j,i}$ and the second relator corresponds to products that bound a hexagon.  Therefore any word that bounds a concatenation of backtracking paths and ``hexagon loops" is group-equivalent to a product of conjugates of relators.  

Define $\rho: G\to \Gamma_{\mathscr{R}(\mathscr{T})}$ by $\beta_{i,j}\mapsto b_{i,j}$.  It is straight forward to check that that $\rho$ is well-defined. 
Let $w=s_1\cdots s_n$ where $s_i\in S$.  Suppose $\rho(w)=1$.  
We need to show that $w=\prod_{m=1}^{\ell} c_mr_mc_m^{-1}$ where $r_m$ is a relator in $G$ and $c_m$ is an element in $G$.

% Define $\rho: \Gamma_{\mathscr{R}(\mathscr{T})} \to \Gamma_\mathscr{T}$ by $b_{i,j} \mapsto a_ia_j^{-1}$.  It is straight forward to check that that $\rho$ is well-defined.  For ease of notation, denote $g_n=b_{i_n,j_n}$.  
% Say $\rho(\prod_{n=1}^k g_n) = 1$.  In order for $\prod_{n=1}^k g_n$ to be the identity in $\Gamma_{\mathscr{R}(\mathscr{T})}$, we need to be able to decompose $\prod_{n=1}^k g_n$ as a product of conjugates of the relators.  
% %
% Let $w=\prod_{n=1}^k g_n$.
% By the definition of a presentation of groups, $w=1$ in $\Gamma_{\mathscr{R}(\mathscr{T})}$ if and only if $w = \prod_{m=1}^\ell c_mr_mc^{-1}_m$ where $r_m$ is a relator and $c_m\in \Gamma_{\mathscr{R}(\mathscr{T})}$.

First, we argue that if for all $k\in \{1, ..., n\}$, $\rho(\prod_{m=1}^{k} s_m)$ (i.e., the entire loop) is in the same apartment, say $A$, then we can decompose it into a product of relators. 
%The first relator just identifies $(b_{i,j})^{-1}$ with $b_{j,i}$.  The second relator corresponds to products that bound a hexagon.  Therefore any word that bounds a concatenation of ``hexagon loops" and backtracking paths is group-equivalent to a product of conjugates of relators.  
% 
Observe that the subgraph of the Cayley graph of $\Gamma_{\mathscr{T}}$ that consists of vertices of the same type as the identity and two-edge paths (with the interim vertex of type $a_i$) corresponding to each $s_m$ gives a hexagonal tessellation of a sub-apartment of $A$.  Any closed loop along a hexagonal tessellation bounds a concatenation of hexagons and backtracking paths.

Therefore, it suffices to write $w$ as a product $\prod_{t=1}^r c_tu_tc_t^{-1}$ where $\rho(u_t)=1$ and $u_t$ is a word in $\{\beta_{i,j}\}$ that traces a loop in a single apartment. That is, $u_t=\prod_{m=1}^{k(t)} s_m^t$ where each $s_m^t\in \{\beta_{i,j}\}$.  Define the images of the initial sub-words of $u_t$ as $u_{t,r}:=\rho(\prod_{m=1}^{r} s_m^t)$ where $r\in \{1, ..., k(t)\}$.  We now have that $\{u_{t,r}: 1\leq r\leq k(t)\}$ is in one apartment $A_t$. 

% and $u_t=\prod_{n=1}^{k(t)} g_n$.  
% Define $u_{t,s}:=\rho(\prod_{n=1}^{s} g_n)$.  For each $t$, we want $\{u_{t,s}: 1\leq s\leq k(t)\}$ to be in one apartment.  

 Let $g_i=\rho(s_i)$.  Then we see that $\rho(w)=\rho(s_1\cdots s_n)=g_1\cdots g_n$ forms a loop in the $\tilde A_2$-building.  Consider the loop in the $\tilde A_2$-building along vertices $g_1=\rho(s_1)$, $g_1g_2=\rho(s_1s_2)$, ..., $g_1\cdots g_n=\rho(s_1\cdots s_n)$ joined by rhombus moves.  
 Note that the minimal path $g_1$ is itself in an apartment. Thus, there exists an $i<j$ such that $g_jg_{j+1}\cdots g_ng_1\cdots g_i$ is contained in an apartment, say $A$, and the length of the complement (i.e., $j-i+1$) is minimized.  If $j=i+1$, then the entire loop is in apartment $A$ and we are done.  

Otherwise, we can connect $g_i$ to $g_j$ by a minimal rhombus path $p$ in $A$ as shown in Figure \ref{fig:rhombus subgroup}.  
\begin{figure}[h!]
\begin{center}   
\begin{tikzpicture}
    
    \draw (0,0) circle (2);

    \begin{scope}[every node/.style={sloped,allow upside down}, label distance=-1mm]

        \draw[->] (-90:2) arc (-90:-100:2) node[below] {$g_j$};
        \draw[->] (200:2) arc (200:190:2) node[left] {$g_n$};
        \draw[->] (180:2) arc (180:169:2) node[left] {$g_1$};
        \draw[->] (110:2) arc (110:100:2) node[above] {$g_i$};
        \draw[blue, ->] (90:2) -- (0:0) node[right] {$p$};
        \draw[blue] (0:0) -- (-90:2);
        \draw[] (-100:2) arc (-100:222:2) node[below,yshift=0.3cm] {$\ddots$};
        \draw[] (160:2) arc (160:135:2) node[above,rotate=85,yshift=-0.1cm] {$\ddots$};
        
        \filldraw (-90:2) circle(2pt);
        \filldraw (-110:2) circle(2pt);
        \filldraw (200:2) circle(2pt);
        \filldraw[red] (180:2) circle(2pt);
        \filldraw (160:2) circle(2pt);
        \filldraw (110:2) circle(2pt);
        \filldraw (90:2) circle(2pt);

    \end{scope}
\end{tikzpicture}
\end{center}
        \caption{}
        \label{fig:rhombus subgroup}
\end{figure}

Now we see that $g_1\cdots g_n = \underbrace{g_1\cdots g_ipg_j\cdots g_n}_{\text{loop in apartment $A$}} \big( \underbrace{g_n^{-1}\cdots g_j^{-1} (\overbrace{p^{-1}g_{i+1}\cdots g_{j-1}}^{\text{element}} )g_j\cdots g_n}_{\text{conjugate of element}}\big)$.  
By \ref{lem:min rhombus path}, we know that such a path exists in every apartment containing $g_i$ and $g_j$.  Consider loop $p^{-1}g_{i+1}\cdots g_{j-1}$.  We can again find $i'<j'$ such that \\$g_{j'}g_{j'+1}\cdots p^{-1}g_{i+1}\cdots g_{i'}$ is contained in an apartment, say $A'$, and the length of the complement is minimized.  Note that $i<i'<j'<j$.  If $j'=i'+1$, then the entire loop is in apartment $A'$ and we are done.  Otherwise can can connect $g_{i'}$ to $g_{j'}$ by a minimal rhombus path $p'$ in $A'$ as shown in Figure \ref{fig:rhombus subgroup2}.  We can repeat this decomposition until the remainder of the original loop is in one apartment.
\end{proof}

\begin{figure}[h!]
\begin{center}   
\begin{tikzpicture}
    \draw (0,0) circle (2);

    \begin{scope}[every node/.style={sloped,allow upside down}, label distance=-1mm]

        \draw[->] (-90:2) arc (-90:-101:2) node[below] {$g_j$};
        \draw[->] (200:2) arc (200:189:2) node[left] {$g_n$};
        \draw[->] (180:2) arc (180:168:2) node[left] {$g_1$};
        \draw[->] (110:2) arc (110:99:2) node[above] {$g_i$};
        \draw[->] (-30:2) arc (-30:-42:2) node[below right] {$g_{j'}$};
        % \draw[->] (30:41:2) arc (30:2) node[above right] {$g_{i'}$};
        \draw[->] (41:2) arc (41:37:2) node[above right] {$g_{i'}$};
        \draw[blue, ->] (90:2) -- (0:0) node[right] {$p$};
        \draw[blue] (0:0) -- (-90:2);
        \draw[] (-100:2) arc (-100:222:2) node[below,yshift=0.3cm] {$\ddots$};
        \draw[] (160:2) arc (160:135:2) node[above,rotate=85,yshift=-0.1cm] 
                {$\ddots$};
         \draw[] (-50:2) arc (-50:-70:2) 
                node[below,rotate=60,yshift=0.4cm,xshift=-0.1cm] {$\ddots$};
        \draw[] (50:2) arc (50:70:2) 
                node[above, rotate=18, yshift=-0.2cm, xshift=0.1cm] {$\ddots$};
        \draw[cyan,->] (50:2) -- (0:1.3) node[right] {$p'$};
        \draw[cyan] (0:1.3) --(-50:2);
       
        \filldraw (-90:2) circle(2pt);
        \filldraw (-110:2) circle(2pt);
        \filldraw (200:2) circle(2pt);
        \filldraw[red] (180:2) circle(2pt);
        \filldraw (160:2) circle(2pt);
        \filldraw (110:2) circle(2pt);
        \filldraw (90:2) circle(2pt);
        \filldraw (50:2) circle(2pt);
        \filldraw (30:2) circle(2pt);
        \filldraw (-50:2) circle(2pt);
        \filldraw (-30:2) circle(2pt);

    \end{scope}
\end{tikzpicture}
\end{center}
        \caption{}
        \label{fig:rhombus subgroup2}
\end{figure}

Similar to the beginning of section 7 in \cite{witzel2017panel}, we see that if each $a_i$ maps non-trivially to the abelianization of $\Gamma_\mathscr{T}$, then there exists a homomorphism $\psi: \Gamma_\mathscr{T} \to \Z/3\Z$ such that $a_i \mapsto 1$ for $0\leq i\leq v-1$.  Note that because $\psi$ is a homomorphism, $\psi(a_i^{-1})=-\psi(a_i)=-1=2$.  Thus, we see that all $b_{*,*}$ are in the kernel of $\psi$. Also, observe that moving a vertex $v$ by three generators of $\Gamma_\mathscr{T}$ either yields the identity if the three generators are linked to a triangle presentation element, or yields another element, say $v'$ of the same type as vertex $v$.  We can modify the action where three generators takes $v$ to $v'$ by using two elements of $b_{*,*}$. Part $a$ of the following proposition follows immediately:

\begin{prop} %{\ } 
        Group $\Gamma_{\mathscr{T}}=
            \ker \psi \rtimes \langle a_i \rangle$ for any generator $a_i$ of $\Gamma_{\mathscr{T}}$. 
\end{prop}

We next make the following observation: 

\begin{obs} \label{obs:}
    Let the chamber triple $(\sigma_0, \sigma_1, \sigma_2)$ also be the fundamental domain of the $\tilde A_2$-building that corresponds to the below Cayley graph triangle.  

\begin{center}   
\tikzset{
  A/.style={regular polygon, regular polygon sides=3,fill=gray!20,minimum size=3cm, draw},
  B/.style={regular polygon, regular polygon sides=3,fill=gray!20,minimum size=3cm, draw},
}
\begin{tikzpicture}[node distance=0pt, every node/.style={outer sep=0pt}]
  %\draw [help lines,gray!50,step =1] (-5,-5) grid (5,5);
  \node (A) [A=0] {};

    \begin{scope}[every node/.style={sloped,allow upside down}]

        \draw [draw=black]
            (A.corner 1) -- node {\midarrow} (A.corner 2) -- node {\midarrow} 
                (A.corner 3) -- node {\midarrow} cycle;
    \end{scope}
    
      \draw[shift=(A.side 2)] node[below] {$a_i$};
      \draw[shift=(A.side 1)] node[below right] {$a_k$};
      \draw[shift=(A.side 3)] node[above right] {$a_j$};
    
      \node[circle, fill=red, inner sep=2pt, label=east:{$a_i$}] at (A.corner 3) {};
      \node[circle, fill=cyan, inner sep=2pt, label={$a_k$}] at (A.corner 1) {};
      \node[circle, fill=blue, inner sep=2pt, label=west:{$e$}] at (A.corner 2) {};

\end{tikzpicture}
\end{center}

    \noindent Then we can understand $\sigma_0$ as $p$, $\sigma_1$ as $a_ipa_i^{-1}$, and $\sigma_2$ as $a_k^{-1}pa_k$ (where $p$ was previously defined as $p(i)=i+1$). Thus $\Gamma''_{\mathscr{A}}=\langle p, \ a_ipa_i^{-1}, \ a_k^{-1}pa_k \rangle$. 
\end{obs}

We can also see that both maps $\phi$ and $\psi$ extend to domain $\Gamma_\mathscr{T} \rtimes \langle p \rangle$.

\begin{prop}
    There exists a map $\Phi: \Gamma_\mathscr{T} \rtimes \langle p \rangle \to \Z/v\Z$ that is an extension of $\phi: \Gamma''_{\mathscr{A}} \to \Z/v\Z$.  There also exists a map $\Psi: \Gamma_\mathscr{T}\rtimes \langle p \rangle  \to \Z/3Z$ that is an extension of $\psi: \Gamma_{\mathscr{T}} \to \Z/3\Z$. 
\end{prop}

\begin{proof}
    
    Define $\Phi: \Gamma_\mathscr{T} \rtimes \langle p \rangle \to \Z/v\Z$ on generators as $p\mapsto 1$ and $a_i\mapsto 0$ for all $0\leq i\leq v-1$.  Map $\Phi$ is well defined:  Say $(i,j,k)\in \mathscr{T}$, then $a_ia_ja_k=1$ and  $\Phi(a_ia_ja_k)=\Phi(a_i)+\Phi(a_j)+\Phi(a_k)=0+0+0=0$.  Recall that $pa_ip^{-1}=a_{i+1}$, then $\Phi(pa_ip^{-1}a_{i+1}^{-1})=\Phi(p)+\Phi(a_i)-\Phi(p)-\Phi(a_{i+1})=1+0-1-0=0$.  Lastly, $\Phi(p^v)=v\Phi(p)=0$ (mod $v$).  
    We also see that $\eval{\Phi}_{\Gamma''_{\mathscr{A}}} =\phi$. Therefore $\Phi$ is an extension of $\phi$.

    Next define $\Psi: \Gamma_\mathscr{T}\rtimes \langle p \rangle  \to \Z/3Z$ on generators as $p\mapsto 0$ and $a_i\mapsto 1$ for all $0\leq i\leq v-1$.  
    Map $\Psi$ is also well defined.
    Say $(i,j,k)\in \mathscr{T}$, then $a_ia_ja_k=1$ and  $\Psi(a_ia_ja_k)=\Psi(a_i)+\Psi(a_j)+\Psi(a_k)=1+1+1=0$ (mod 3).   
    Also, $\Psi(pa_ip^{-1}a_{i+1}^{-1})=\Psi(p)+\Psi(a_i)-\Psi(p)-\Psi(a_{i+1})=0+1-0-1=0$.  Lastly, $\Psi(p^v)=v\Psi(p)=v\cdot0=0$. Moreover, we see that $\eval{\Psi}_{\Gamma''_{\mathscr{A}}} =\psi$.  Therefore, $\Psi$ is an extension of $\psi$.  
\end{proof}

We use the above proposition to prove the following:
\begin{thm}
    $\ker \phi=\ker \psi=\Gamma_{\mathscr{R}(\mathscr{T})}$.  
\end{thm}

\begin{proof}
    From Observation \ref{obs:}, we have that $\sigma_0=p$, $\sigma_1=a_ipa_i^{-1}$, and $\sigma_2=a_k^{-1}pa_k$.  Thus,  $\phi(p)=\phi(a_ipa_i^{-1})=\phi(a_k^{-1}pa_k)=1$ and $\ker \Phi=\Gamma_{\mathscr{T}}$.  
    Also, $ker \Psi$ consists of the type-preserving elements of $\Gamma_{\mathscr{T}}\rtimes \langle p \rangle$, which makes $\ker \Psi = \Gamma''_{\mathscr{A}}$.  Moreover, $\phi = \eval{\Psi}_{\Gamma''_{\mathscr{A}}}$ because $\phi=\Phi$ on the generators of $\Gamma''_{\mathscr{A}}$. 
    This makes $\ker \phi = \ker \Phi \cap \Gamma''_{\mathscr{A}} = \ker \Phi \cap \ker \Psi$.  
    Now by definition $\psi=\eval{\Psi}_{\Gamma_\mathscr{T}}$.  This makes $\ker \psi = \Gamma_{\mathscr{T}} \cap \ker \Psi = \ker \Phi \cap \ker \Psi$.  
    Therefore, $\ker \phi = \ker \psi$.  By Lemma \ref{lem:rhombus generators}, we see that $\ker \psi=\Gamma_{\mathscr{R}(\mathscr{T})}$.  
\end{proof}

We can now flesh out Figure \ref{fig:poset} as shown:

{
\begin{figure}[h!]
\begin{center}    
    \begin{tikzcd}[column sep=0em,row sep=2em]
        & \substack{\tilde \Gamma_{\mathscr{T}} = \\
                (\Gamma_{\mathscr{T}} \rtimes \langle p \rangle) \rtimes \langle s\rangle = \\ (\Gamma''_{\mathscr{A}}\rtimes \langle a_i \rangle)\rtimes \langle s \rangle}
           \\
        & \substack{\Gamma_{\mathscr{T}} \rtimes \langle p \rangle= \\ \Gamma''_{\mathscr{A}}\rtimes \langle a_i \rangle}
                \arrow[u, swap, "\text{extend by $\langle s \rangle$}"]\\
          \substack{\Gamma_{\mathscr{T}}= \\ \ker \Phi = \\ \ker \psi\rtimes \langle a_i \rangle}
                \arrow[ur, "\text{extend by $\langle p \rangle$}"] 
                 && 
          \substack{\text{Singer lattice } \Gamma''_{\mathscr{A}} = \\ \ker \Psi = \\ \ker \phi \rtimes \langle p \rangle}
                \arrow[ul, swap, "\text{extend by $\langle a_i \rangle$}"] \\
        & \substack{\ker \phi = \\ \ker \psi = \\ \Gamma_{\mathscr{R}(\mathscr{T})}}
                \arrow[ul,"\text{extend by $\langle a_i \rangle$}"] 
                \arrow[ur, swap, "\text{extend by $\langle p \rangle$}"] 
        
    \end{tikzcd}    
\end{center}
    \caption{Fleshed out Poset of Lattices}
    \label{fig:poset}
\end{figure}
}

\newpage
\hfill\\

\section{Examples} \label{sec:examples}

\par The below examples demonstrate how triangle presentations of order $q$ are encoded by perfect difference fixed by multiplier $q$. We have color-coded the orbits of the \PDS s and used the same coloring with the corresponding triangle presentations.  

In the below table, ``PP" stands for ``Projective Plane."  Recall that the order of the projective plane is $q$ and the number of elements in the \PDS\ is 
$q+1$. \\ \\

\begin{longtable}{CLL}
\text{PP order} & \text{Perfect Difference Set} & 
       \text{Triangle Presentation} \\
\hline
\\
2 & \{{\color{Orange}1},{\color{Orange}2},{\color{Orange}4}\} 
        & \big\{ \, \langle {\color{Orange}1}, {\color{Orange}2},
            {\color{Orange}4} \rangle \, \big\} \\ \\
3 & \{0,{\color{Orange}1},{\color{Orange}3},{\color{Orange}9}\}
        & \big\{ \, \langle {\color{Orange}1},  {\color{Orange}3},
            {\color{Orange}9} \rangle , \  
            \langle 0, 0, 0 \rangle \, \big\} \\ \\
4 & \{0,{\color{Orange}1},{\color{Orange}4},{\color{blue}14},{\color{Orange}16}\}
        & \big\{ \, \langle {\color{Orange}1},  {\color{Orange}4},  
            {\color{Orange}16} \rangle , \ 
            \langle {\color{blue}14},  {\color{blue}14},{\color{blue}14} \rangle, \\
        & & \ \ \langle 0, 0, 0 \rangle \, \big\} \\ \\
5 & \big\{ {\color{Orange}1}, {\color{Orange}5}, {\color{blue}17},
            {\color{blue}22}, {\color{blue}23}, {\color{Orange}25}  \big\}
        & \big\{ \, \langle {\color{Orange}1},  {\color{Orange}5},  
            {\color{Orange}25} \rangle , \ 
            \langle {\color{blue}17},  {\color{blue}23},{\color{blue}22} \rangle \, \big\} \\ \\
7 & \big\{ 0, {\color{Orange}1}, {\color{blue}5}, {\color{Orange}7}, 
            {\color{blue}17}, {\color{blue}35}, {\color{ForestGreen}38}, {\color{Orange}49} \big\}
        & \big\{ \, \langle {\color{Orange}1}, {\color{Orange}7}, 
            {\color{Orange}49} \rangle , \
            \langle {\color{blue}5}, {\color{blue}35}, {\color{blue}17} \rangle, 
            \\
        & & \ \ \langle {\color{ForestGreen}38}, {\color{ForestGreen}38},   
            {\color{ForestGreen}38} \rangle ,  \
            \langle 0, 0, 0 \rangle \, \big\} \\ \\
8 & \big\{ {\color{Orange}1}, {\color{blue}2}, {\color{ForestGreen}4}, 
            {\color{Orange}8}, {\color{blue}16}, {\color{ForestGreen}32}, {\color{ForestGreen}37}, {\color{blue}55}, {\color{Orange}64} \big\}
        & \big\{ \, \langle {\color{Orange}1}, {\color{Orange}8}, 
            {\color{Orange}64} \rangle , \
            \langle {\color{blue}2}, {\color{blue}16}, {\color{blue}55} \rangle, 
            \\
        & & \ \ \langle {\color{ForestGreen}4}, {\color{ForestGreen}32},       
            {\color{ForestGreen}37} \rangle \, \big\} \\ \\
9 & \big\{ 0, {\color{Orange}1}, {\color{blue}3}, {\color{Orange}9}, 
            {\color{blue}27}, {\color{ForestGreen}49}, {\color{ForestGreen}56}, {\color{blue}61}, {\color{ForestGreen}77}, {\color{Orange}81}  \big\}
        & \big\{ \, \langle {\color{Orange}1}, {\color{Orange}9},
            {\color{Orange}81} \rangle, \
            \langle {\color{blue}3}, {\color{blue}27}, {\color{blue}61} \rangle, 
            \\
        & & \ \ \langle {\color{ForestGreen}49}, {\color{ForestGreen}77},
            {\color{ForestGreen}56} \rangle, \
            \langle 0, 0, 0 \rangle \, \big\} \\ \\
11 & \big\{ {\color{Orange}1}, {\color{blue}10}, {\color{Orange}11},
            {\color{blue}13}, \ {\color{ForestGreen}27}, {\color{ForestGreen}31}, {\color{DarkOrchid}68}, \ {\color{ForestGreen}75}, {\color{DarkOrchid}83}, 
        & \big\{ \, \langle {\color{Orange}1}, {\color{Orange}11}, 
            {\color{Orange}121} \rangle, \ 
            \langle {\color{blue}10}, {\color{blue}110}, {\color{blue}13} \rangle, 
            \\
        & \ \ {\color{blue}110}, {\color{DarkOrchid}115}, {\color{Orange}121}  
            \big\} 
        & \ \ \langle {\color{ForestGreen}27}, {\color{ForestGreen}31},     
            {\color{ForestGreen}75} \rangle ,  \
            \langle {\color{DarkOrchid}68}, {\color{DarkOrchid}83}, {\color{DarkOrchid}115} \rangle \, \big\} \\ \\
13 & \big\{ 0, {\color{Orange}1}, {\color{Orange}13}, {\color{Cerulean}61}, 
            {\color{DarkOrchid}67}, {\color{ForestGreen}69}, {\color{blue}107}, {\color{blue}110}, {\color{ForestGreen}132},  
        & \big\{ \, \langle {\color{Orange}1}, {\color{Orange}13}, 
            {\color{Orange}169} \rangle, \
            \langle {\color{blue}107}, {\color{blue}110}, {\color{blue}149} \rangle,  \\
        & \ \ {\color{DarkOrchid}139}, {\color{blue}149}, \ 
            {\color{DarkOrchid}160}, \ {\color{ForestGreen}165}, \ {\color{Orange}169}  \big\}
        & \ \ \langle {\color{ForestGreen}69}, {\color{ForestGreen}132},    
            {\color{ForestGreen}165} \rangle , \
            \langle {\color{DarkOrchid}67}, {\color{DarkOrchid}139}, {\color{DarkOrchid}160} \rangle, \\
        & & \ \ \langle {\color{Cerulean}61}, {\color{Cerulean}61}, 
            {\color{Cerulean}61} \rangle,  \ 
            \langle 0, 0, 0 \rangle \, \big\} \\ \\

\end{longtable}

%\hfill\\

\bibliographystyle{alpha}
\bibliography{refs}

\end{document}